\documentclass{amsart}
\usepackage{amsmath,amsthm,verbatim,amssymb,amsfonts,amscd,blindtext,  graphicx,tensor,chngcntr}
\usepackage{xcolor}
\usepackage[utf8x]{inputenc}
\usepackage{mathtools}
\usepackage{arcs}
\usepackage{graphics}

\usepackage{euscript, enumerate}
\usepackage{url,hyperref}

\newcommand{\R}{{\mathbb R}}

\newcommand{\p}{\partial}
\newcommand{\fr}{\frac}
\newcommand{\la}{\langle}
\newcommand{\ra}{\rangle}
\newcommand{\na}{\nabla}

\newcommand{\e}{\epsilon}

\newcommand{\be}{\begin{equation}}
\newcommand{\ba}{\begin{aligned}}
\newcommand{\bee}{\begin{equation*}}
\newcommand{\ee}{\end{equation}}
\newcommand{\ea}{\end{aligned}}
\newcommand{\eee}{\end{equation*}}

\newcommand{\ds}{\displaystyle}

\theoremstyle{plain}
\newtheorem{theorem}{Theorem}[section]

\newtheorem{corollary}[theorem]{Corollary}
\newtheorem{cor}[theorem]{Corollary}
\newtheorem{lemma}[theorem]{Lemma}

\newtheorem{prop}[theorem]{Proposition}

\newtheorem{claim}{Claim}[section]

\theoremstyle{remark}

\newtheorem{remark}[theorem]{Remark}

\theoremstyle{definition}
\newtheorem{definition}[theorem]{Definition}

\numberwithin{equation}{section}
\begin{document}
\title{Convergence of  Gauss curvature flows to translating solitons}
\author{Beomjun Choi}
\address{{\bf Beomjun Choi:} 
{Department of Mathematics, POSTECH, 77 Cheongam-Ro, Nam-Gu, Pohang, Gyeongbuk 37673, Republic of Korea}}
\email{{bchoi@postech.ac.kr}}
\author{Kyeongsu Choi}
\address{{\bf Kyeongsu Choi:}  School of Mathematics, Korea Institute for Advanced Study, 85 Hoegiro, Dongdaemun-gu, Seoul 02455, Republic of Korea}
\email{choiks@kias.re.kr}
\author{Panagiota Daskalopoulos}
\address{{\bf Panagiota Daskalopoulos:} Department of Mathematics, Columbia University, 2990 Broadway, New York,  NY 10027, USA}
\email{pdaskalo@math.columbia.edu}

\begin{abstract} We address the asymptotic behavior of the $\alpha$-Gauss curvature flow,  for $\alpha >1/2$,  with  a complete non-compact convex {initial} hypersurface which is contained in a cylinder of a bounded cross section. We show that  the flow converges, as $t \to +\infty$,  locally smoothly to a translating soliton which is uniquely determined by the  asymptotic cylinder of  the initial hypersurface.

\end{abstract}
\maketitle
\tableofcontents

\section{Introduction}
Given $\alpha>0$, the $\alpha$-Gauss curvature flow ($\alpha$-GCF in abbreviation) is a one-parameter family of embeddings $F:M^n\times [0,T) \to \mathbb{R}^{n+1}$ such that for each $t\in [0,T)$, $F(M^n,t)=\Sigma_t$ is a complete convex hypersurface in $\mathbb{R}^{n+1}$, and $F(\cdot,t)$ satisfies
\begin{equation}
\tfrac{\partial}{\partial t}F(p,t)=-K^\alpha(p,t) \nu (p,t).
\end{equation}
{Here,} $K(p,t)$ is the Gauss curvature of $\Sigma_t$ at $F(p,t)$, and $\nu(p,t)$ is the unit normal vector of $\Sigma_t$ at $F(p,t)$ pointing  outward of the convex hull of $\Sigma_t$.

\medskip

The classical Gauss curvature flow (GCF), the $\alpha=1$ case, was first introduced by W. Firey \cite{Fi} to describe the shape of worn stones and the asymptotic behavior when it disappears. In \cite{Fi}, W. Firey proved that if a closed strictly convex solution to the GCF in $\mathbb{R}^3$ has the central symmetry, then it converges to a round sphere after rescaling. Later, B. Andrews \cite{An99} removed the central symmetry condition. In higher dimensions $n \geq 3$, P. Guan and L. Ni \cite{GN} obtained the convergence to a self-shrinking soliton after rescaling, and K. Choi and P. Daskalopoulos \cite{CD} showed the uniqueness of self-shrinking {solitons}. Namely, a closed strictly convex solution to the GCF  in $\mathbb{R}^{n+1}$ converges to a round sphere after rescaling.

\medskip

In addition to the {classical} case $\alpha=1$, the asymptotic behavior of the  $\alpha$-GCF {also} has been widely studied. 
In particular, in the $\alpha=\tfrac{1}{n+2}$ case, an affine transform of a solution  remains as a solution, and thus we call the $\tfrac{1}{n+2}$-GCF as the affine normal flow. E. Calabi \cite{Ca72} showed that a self-shrinking soliton to the affine normal flow is an ellipsoid. (See also \cite{BCD} for an alternative proof.) B. Andrews \cite{An96} obtained the convergence of the closed affine normal flow to an ellipsoid after rescaling.

In the range of $\alpha> \tfrac{1}{n+2}$, the convergence of the closed $\alpha$-GCF to a round sphere after rescaling has been shown by B. Chow \cite{Ch85} for $\alpha=\tfrac{1}{n}$, and by B. Andrews and X. Chen \cite{AC} for $\tfrac{1}{2}\leq \alpha \leq 1$ and $n=2$. Later, for the all $\alpha >\tfrac{1}{n+2}$ B. Andrews, P. Guan and L. Ni \cite{AGN} showed the convergence to a self-similar soliton after rescaling. Moreover S. Brendle, K. Choi, and  P. Daskalopoulos \cite{BCD} proved the uniqueness of self-shrinking solitons. Namely{, for $\alpha>\tfrac{1}{n+2}$,} a closed strictly convex solution to the $\alpha$-GCF {in $\mathbb{R}^{n+1}$} converges to a round sphere after rescaling.

{In the range of} small powers $\alpha \in (0,\frac{1}{n+2})$, the asymptotic behavior remains as an open problem. B. Andrews classified closed self-shriking solitons in the curve case $n=1$  \cite{An03}, and showed the existence of non-trivial closed self-shrinking solitons in higher dimensions  \cite{An00}.

\bigskip

Regarding the non-compact case, the translating solitons to the $\alpha$-GCF have been classified for $\alpha=\tfrac{1}{n+2}$ and $\alpha >\tfrac12$. In the affine normal case $\alpha=\tfrac{1}{n+2}$, the translating solitons are paraboloids. The $n=2$ case showed first by K. J\"orgens \cite{Jo}, and later by  J.C.C. Nitsche \cite{Ni} with another proof by using the complex analysis. E. Calabi \cite{Ca58} extended the result for $n\leq 5$, and A.V. Pogorelov \cite{Po} proved for all dimensions. S.Y. Cheng and S.T. Yau \cite{CY} provided an alternative proof by using the affine geometry. {See also the recent classification result \cite{CCK} of K.Choi, B.Choi and S.Kim for the case $n=2$ and $\alpha<\frac{1}{4}$.} 

 \medskip

In \cite{Ur1,Ur2}, J. Urbas  showed that  every translating soliton for $\alpha>\tfrac12$ is contained in a bounded cylinder {$\overline{\Omega} \times \mathbb{R}$\footnote{In this paper, $\Omega\subset \mathbb{R}^n$ denotes an open set.}}, namely $\Omega \subset \mathbb{R}^n$ is bounded. Moreover, if $\alpha >\tfrac12$ then given a {bounded convex body\footnote{In this paper, we say that $\mathcal{K}$ is a bounded convex body if it is a compact convex set with non-empty interior. In addition, an unbounded convex body means an unbounded closed convex set with non-empty interior.} $\overline{\Omega}\subset \mathbb{R}^n$} there exists a translating soliton asymptotic to {$ \partial\Omega \times \mathbb{R}$}. Furthermore, for each {bounded convex body  $\overline{\Omega}$}, the translating soliton is unique up to translations. One the other hand, for {small powers} $\alpha \in (0,\tfrac12]$, H. Jian and X.J. Wang \cite{JW} showed the existence of infinitely many entire translating solitons.

\medskip

Recently the authors \cite{CCD} showed the convergence to a translating soliton for $n=1$ and $\alpha>\tfrac12$. In this paper, we establish  its higher dimensional result for $n \geq 2$ as follows. 
{
\begin{theorem}\label{thm-main} Let $ \mathcal{K}_0\subset\mathbb{R}^{n+1}$ be an unbounded convex body asymptotic to a convex cylinder $\overline{\Omega}\times \mathbb{R} (\neq \mathcal{K}_0)$ with the bounded section $\overline{\Omega}\subset\mathbb{R}^n$. Then, given $\alpha \geq 1$, the viscosity solution\footnote{See Definition \ref{def-weaksol}.} to the $\alpha$-Gauss curvature flow from the initial hypersurface $\Sigma_0=\partial\mathcal{K}_0$  locally smoothly converges   to the translating soliton asymptotic to $\partial \Omega \times \mathbb{R}$ as $t \to +\infty$.
\end{theorem}
}

\medskip

\textbf{Local convergence}:  The viscosity flow $\Sigma_t$ is asymptotic to the  initial asymptotic cylinder, say $\partial\Omega \times \mathbb{R}$, for all time by Theorem \ref{thm-exist1} and thus $\Sigma_t$ can be written as  convex graphs on a fixed domain\[\Sigma_t = \p \, \{ x_{n+1} > u(x,t)\,: \, x\in  \Omega  \}. \] {The local smooth} convergence in the statement of the above theorem {implies} the $C^{\infty}_{loc}(\Omega)$ convergence of the functions $u(\cdot,t) - \inf_{x\in \text{int}(\Omega)} u(x,t)$ to {$u_{\Omega}\in C^{\infty}(\Omega)$}, which  represents the {translating} soliton asymptotic to {$\partial\Omega\times \mathbb{R}$. If the weakly convex domain $\Omega$ is not strictly convex,} then the corresponding translating soliton may touch  the boundary of the cylinder  and have flat sides. {(See the work  by K. Choi,  P. Daskalopoulos, and K.A. Lee in \cite{CDL}.)} Therefore,  {the} smooth convergence up to boundary is not expected.  

\bigskip

\textbf{Viscosity solution}: We introduce the notion of {the} viscosity solutions to  $\alpha$-GCF in Definition \ref{def-weaksol} to state and prove the convergence of flows from weakly convex {non-smooth} initial hypersurfaces. The existence and  uniqueness of the  viscosity flow is shown in  Theorem \ref{thm-exist1}. Note that if $\Sigma_0$ is weakly convex and has flat sides, the solution $\Sigma_t$ preserves the flat sides for a certain amount of time by the result of R. Hamilton \cite{Ha94}. See also the optimal regularity of an evolving flat side for  short time \cite{DH} and for long time \cite{DL}. Regardless of the regularity of $\Sigma_0$,  for each $ \Omega'\subset\subset  \Omega $ we show that the flow $\Sigma_t$  becomes smooth and strictly convex  in  $ \Omega'\times \mathbb{R}$ for large time $t\gg 1$ and smoothly converges to the translating soliton. {In our subsequential work \cite{CCD2}, we show the uniqueness of ancient solutions which are asymptotic to a convex cylinder $\partial  \Omega \times \mathbb{R}$ and we use Theorem \ref{thm-main}, 
with the notion of viscosity solution,  in a crucial way. Indeed, some ancient solutions are not of class $C^2$ \cite{CDL,CCD2}.}

\bigskip

{ \textbf{Additional steps in higher dimensions}: Compared to the $n=1$ case \cite{CCD}, the entropy formulas here become more involved so we provide them in Appendix. Moreover, for the local smooth convergence, one needs to  establish local upper and lower bounds on the principal curvatures, which are independent of the regularity of the initial data $\Sigma_0=\partial \mathcal{K}_0$. 
Since the  linearized operator highly degenerates along horizontal directions,  in Section \ref{sec:local_convex} we introduce some geometric ideas and establish new estimates.
}
 
 \bigskip

 For small $\alpha \in (\frac{1}{2} ,1)$ {the same result holds under} the technical assumption  that {$\Sigma_0$} can be approximated by {closed} hypersurfaces with uniform bounds for ${\int K^\alpha\,  dg}$ and ${ (\alpha-1)\int PK^\alpha \, dg}$, where $P$ is defined at \eqref{eq-def pij}.
Notice that $\int K^\alpha dg$ and $(\alpha-1)\int PK^\alpha dg$ denote  the total speed and total acceleration, respectively. See Lemma \ref{lem-firstderivative}. Our result for $\alpha \in (\frac{1}{2},1)$ states as follows: 

{
\begin{theorem} \label{thm-main2} Let $ \mathcal{K}_0\subset\mathbb{R}^{n+1}$ be an unbounded convex body asymptotic to a convex cylinder $\overline{\Omega}\times \mathbb{R} (\neq \mathcal{K}_0)$ with  bounded section $\overline{\Omega}\subset\mathbb{R}^n$. Suppose that given $\alpha \in (\frac{1}{2},1)$, there is a sequence of bounded strictly convex bodies $\mathcal{K}_i$ with smooth boundaries $\Sigma_i=\p\mathcal{K}_i$ which increases to $\Sigma_0=\partial \mathcal{K}_0$ \rm (\em i.e. $\mathcal{K}_{k}\subset \mathcal{K}_{k+1}$ and $ {\partial }( {\cup_i \mathcal{K}_i}) = {\partial }\mathcal{K}_0$\rm ) \em with uniform upper bounds for $\int_{\Sigma_i} K^\alpha dg$ and $(\alpha-1)\int_{\Sigma_i} PK^\alpha dg$. Then, the viscosity solution to the $\alpha$-Gauss curvature flow converges  locally smoothly to the translating soliton asymptotic to $\partial\Omega\times \mathbb{R}$ as   $t \to +\infty$. 
\end{theorem}

 \medskip

  \textbf{Convergence with small $\alpha \in (\frac{1}{2},1)$}: Since we have $ K  dg=  d\text{vol}_{\mathbb{S}^n}$ under the Gauss map,  upper bounds for $\int_{\Sigma} K^\alpha dg=\int_{\mathbb{S}^n}K^{\alpha-1}d\text{vol}_{\mathbb{S}^n}$ with $\alpha<1$ are related to local lower bounds for $K$. In the one-dimensional case \cite{CCD}, the local lower bounds for the curvature $\kappa$ were obtained by considering the evolution equation of $\kappa$ as a fast diffusion equation. However, in higher dimensions the Gauss curvature $K$ is not a solution to a porous medium equation any more, and thus it is hard to derive lower bounds for $K$. 
The convergence for $\alpha \in (\frac{1}{2},1)$ without the technical assumption of the bounded total speed and acceleration poses an interesting question that remained to be addressed. 

}
\bigskip

Let us remark  that in order to converge to a translating soliton, the initial hypersurface $\Sigma_0$ must  be  contained in a bounded cylinder.  Jointly with L. Kim and K.A. Lee, the second and third authors in  \cite{CDKL} showed by a barrier argument that if $\Sigma_0$ is a graph over a (possibly non-compact) domain $\Omega_0 \subset\mathbb{R}^n$, then any solution $\Sigma_t$ running from $\Sigma_0$ must remain as a graph over the same domain $\Omega_0$. On the other hand, every translating soliton for the $\alpha$-GCF with $\alpha> 1/2$ is asymptotic to a cylinder of a bounded cross {section} by \cite{Ur1,Ur2}. Hence,  it is necessary to assume that $\Sigma_0$ is contained in a bounded cylinder.

\bigskip

 The following monotonicity formula will be  used to identify the limit as a soliton.  The technical assumptions in Theorem \ref{thm-main2} were made so that this inequality can be applied. 
\begin{theorem}\label{thm-main3} Given $\alpha \geq \tfrac{n-1}{2n}$, compact strictly convex smooth solution $\Sigma_t$ to the $\alpha$-GCF satisfies\[\fr{d}{d t} \int_{\Sigma_t}  P K^\alpha dg \geq \left( n^{-1}  +2\alpha-1\right)\int_{\Sigma_t} P^2K^\alpha dg \geq 0.\]
\end{theorem}
We notice that B. Chow \cite{Ch91} obtained the above monotonicity formula for the GCF ($\alpha=1$); (see the proof of Lemma 4.3 in \cite{Ch91}). In the same paper, B. Chow also obtained a monotonicity formula (Lemma 5.2 in \cite{Ch91}) for the rescaled GCF. In \cite{An94} B. Andrews generalized the monotonicity formula for the rescaled $\alpha$-GCF. {Although Theorem \ref{thm-main3} is a straightforward generalization of \cite{Ch91}, the formula seems not to be shown or used before, so we prove it in   Appendix.}
\bigskip

\bigskip 

\section{Preliminaries}

\begin{definition} (i) $\Sigma\subset \mathbb{R}^{n+1}$ is a convex hypersurface if it is {the} boundary of {a convex body $\mathcal{K}$, which is either bounded or unbounded.} Notice that {the} convex hypersurface $\Sigma=\p\mathcal{K}$ is complete and embedded. 

(ii) For {a} $C^2$ convex hypersurface $\Sigma=\p\mathcal{K}$, we say it is strictly convex at $p\in\Sigma$ if the second fundamental form with respect to the inner normal is positive definite.

\end{definition}

Throughout this paper,  $h_{ij}$ denotes  the second fundamental form.  For a strictly convex solution, one may  consider the  inverse $b^{ij}$
of the second fundamental form $h_{ij}$,  which satisfies   $b^{ik}h_{kj} =\delta^i_j$. We also denote by $dg:=\sqrt {\det g } \, dx$  the volume form induced from the ambient Euclidean metric. let $S:= \la F, \nu\ra$ and $S_{x_0}:= \la F-x_0, \nu\ra$ denote the support functions with respect to the origin and $x_0 \in \R^{n+1}$, respectively. Moreover, we recall the following tensor $P_{ij}$ and the quantity $P$ defined by  B. Chow in \cite{Ch91}:
\begin{equation}\label{eq-def pij}
P_{ij} := \nabla^2_{ij} K^\alpha -b^{mn}\nabla_m h_{ij} \nabla_n K^\alpha + K^\alpha h^k_ih_{kj}\quad\text{ and }\quad P:=b^{ij}P_{ij}.
\end{equation}

Note that, for solutions to {the} $\alpha$-GCF, \eqref{eq-speed} implies \begin{equation}\label{eq-def p}P=\fr{1}{ \alpha K^\alpha}(\p_t K^\alpha - b^{ij}\nabla_iK^\alpha \nabla_j K^\alpha).\end{equation}

Let us recall the unique existence of translating solitons by J. Urbas and {state the result in  the way we will use in work. }
\begin{definition} [Theorem of J.Urbas \cite{Ur1, Ur2}] \label{def-u_Omega} For $\alpha>1/2$ and a given bounded convex domain $\Omega\subset \mathbb{R}^n$,   let $u_{\Omega}: \Omega \to \mathbb{R}$ denote  the graph function of the unique translating soliton which is asymptotic to $\p \Omega \times \mathbb{R}$, it moves  in the positive $e_{n+1}$ direction,   and satisfies  $\inf u_\Omega =0$. In other words, the hypersurface given  by $\p \{ (x',x_{n+1})\in\mathbb{R}^{n+1}\,:\, x_{n+1} > { u_{\Omega}(x')}\}$ defines  the translating soliton. 
\end{definition}

\begin{remark}[The  result by Urbas in \cite{Ur2}]\label{remark-soliton}
In the case where $\Omega$ is not a strictly convex domain, it  is possible that   $\limsup_{x\to x_0} u_\Omega(x')<\infty$,  for some $x_0\in \p\Omega$.
{Hence}  the hypersurface    $\{x_{n+1}=u_\Omega(x')\}$ {is not necessarily complete.} This is the reason why  in the definition above we defined 
the translating soliton as  $\p \{ (x',x_{n+1})\in\mathbb{R}^{n+1}\,:\, x_{n+1} > {u_{\Omega} (x')}\}$.  Urbas \cite{Ur2} showed the  existence of such solitons and their  uniqueness 
among solutions realized in {a} certain  generalized sense.  To be more specific,  Urbas \cite{Ur2} showed that  if a convex function $u(x')$ defined  on $\Omega$ satisfies the translating soliton equation \begin{equation}
\det D^2u= \beta \, (1+|Du |^2)^{\fr{n+2}{2}-\fr{1}{2\alpha}}
\end{equation}
for some $\beta>0$ in the sense of Alexandrov,  and $|\mathbb{R}^n-  Du(\Omega)|=0$, then $u=u_\Omega +C$, for some  constant $C$. {\em We will use this characterization  of {solitons} in the proof  of Theorem \ref{thm-main}. }
\end{remark}

\begin{definition} \label{def-lambda}For $\alpha>1/2$ and {a} given convex bounded domain $\Omega \subset \mathbb{R}^n$, let us note the speed of {the} associated translating soliton by \be\lambda := \fr{1}{|\Omega|^\alpha}\left[\int_{\mathbb{R}^n} \fr{1}{(\sqrt{1+|p|^2})^{n+2-\fr{1}{\alpha}}}dp \right]^\alpha .\ee 
(The derivation of this formula follows from  \eqref{eq-speedlambda} and $Du(\Omega)=\R^n$). 
{Moreover, note} that when  $\alpha=1$, \be\label{eq-lambda}\lambda :=  \fr{1}{|\Omega|}\left[\int_{\mathbb{R}^n} \fr{1}{(\sqrt{1+|p|^2})^{n+1}}dp \right]= \frac{\omega_n}{2|\Omega|}\ee holds, where $\omega_n=|\mathbb{S}^n|$.
\end{definition}

We derive {the evolution equations of basic geometric quantities.}

\begin{prop} 
For strictly convex {hypersurfaces}, we have 
\begin{align}
\label{eq-nablaK}&\nabla_m K = K b^{ij}\nabla_m h_{ij}\\
\label{eq-divfree}&\nabla_i (b^{ij}K) =0\\
\label{eq-nablab}&\nabla_l b^{ij} = -b^{ip}\nabla_l h_{pq} b^{qj}.
\end{align}
For smooth strictly convex {solutions to the} $\alpha$-GCF, we have 
\begin{align}
\label{eq-metric} &\p_t g_{ij} = -2K^\alpha h_{ij}\\
\label{eq-volume} &\p_t dg = -K^\alpha Hdg\\
\label{eq-normal} &\p_t \nu = \nabla K^\alpha=\nabla_i K^\alpha \nabla^i F\\
\label{eq-hij} &\p_t h_{ij}= \nabla^2_{ij}K^\alpha-K^\alpha h_{ik}h^{k}_j\\
&\qquad = \alpha K^\alpha b^{rs}\nabla^2_{rs} h_{ij} +\alpha K^\alpha (\alpha b^{kl}b^{mn} - b^{km}b^{ln}) \nabla_i h_{mn}\nabla_j h_{kl} +\alpha K^\alpha H h_{ij} -(1+n\alpha) K^\alpha h_{ik}h^{k}_j \\
\label{eq-bij} &\p_t b^{pq} = \alpha K^\alpha b^{ij}\nabla^2_{ij} b^{pq} -\alpha K^\alpha b^{ip}b^{jq} (\alpha b^{kl}b^{mn} + b^{km} b^{ln} ) \nabla_i h_{kl} \nabla_j h_{mn} - \alpha K^\alpha H b^{pq} +(1+n\alpha)K^\alpha g ^{pq} \\
\label{eq-speed} &\p_t K^\alpha = \alpha K^\alpha b^{ij} \nabla^2_{ij}K^\alpha +\alpha HK^{2\alpha}\\
\label{eq-distsq} &\p_t |F|^2= \alpha K^\alpha b^{ij}\nabla^2_{ij} |F|^2 +2(n\alpha-1)K^\alpha S -2 K^\alpha b^{ij}g_{ij}\\
\label{eq-support} &\p_t S=\alpha K^\alpha b^{ij} \nabla^2_{ij} S+ \alpha K^\alpha HS - (1+n\alpha )K ^\alpha.
\end{align}
\begin{proof}
By $K=(\det  g^{ij}) (\det h_{ij})$
\[\ba\nabla_m K=K\nabla_m \log K=K\nabla_m \log (\det h_{ij})=Kb^{ij}\nabla_m h_{ij}.\ea\] 
Next, 
\[\ba \nabla_i (b^{ij}K) &=( \nabla_i b^{ij}) K + b^{ij} \nabla_i K 
=-b^{ik}b^{jl} (\nabla_{i} h_{kl})K +b^{ij} K b^{kl} (\nabla_i h_{kl}) =0. \ea\]
{The identity} \eqref{eq-nablab} follows from taking a derivative on $b^{ij}h_{jk}=\delta{i}_k$. 
The evolution equations  \eqref{eq-metric} - \eqref{eq-speed} are shown in \cite[Proposition 2.1]{CDKL}. Note that \[\ba(\p_t  -\alpha K^\alpha b^{ij}\nabla^2_{ij} )F &=-K^\alpha \nu -\alpha K^\alpha b^{ij} h_{ij} (-\nu)\\
&= (n\alpha -1)K^\alpha \nu.\ea\] Thus, we have 
\[\ba (\p_t -\alpha K^\alpha b^{ij})\la F,F\ra &= 2\la F, (n\alpha-1)\nu\ra -2\alpha K^\alpha b^{ij} \la \na_i F, \na_j F\ra \\
&= 2(n\alpha-1)S -2K^\alpha b^{ij}g_{ij}\ea\]and, using $\na^2_{ij}\nu = \na_i (h_{jk} \na^k F )= -h_{jk}h^{k}_i \nu +\na_kh_{ij} \na^k F$,
we obtain 
\[\ba (\p_t -\alpha K^\alpha b^{ij} \nabla^2_{ij})\la F,\nu \ra &=(n\alpha-1)K^\alpha + \la F,(\p_t -\alpha K^\alpha b^{ij}\nabla^2_{ij})\nu\ra - 2\alpha K^\alpha b^{ij} \la \nabla_i F, \nabla_j \nu \ra \\
&= (n\alpha-1)K^\alpha +\la F, \nabla K^\alpha - \alpha K^\alpha b^{ij}\nabla_k h_{ij} \nabla^k F +\alpha K^\alpha H\nu \ra -2n\alpha K^\alpha \\
&= -(n\alpha -1)K^\alpha +\alpha K^\alpha HS. \ea\]

\end{proof}
\end{prop}

{Let us next introduce the following definition of viscosity solutions that we  will employ throughout this work. Similar  definitions  have been frequently used in the literature, for instance in \cite{An00,AMZ}.}
 
\begin{definition}[viscosity solution]\label{def-weaksol}
Let $\mathcal{K}_t \subset \mathbb{R}^{n+1}, t \in [0,T),$ be  a {continuous} one-parameter family of convex bodies which are either bounded or unbounded. $\Sigma_t =\p \mathcal{K}_t \subset \mathbb{R}^{n+1}, t \in [0,T)$,  is a {\em viscosity subsolution}  to the $\alpha$-GCF   if  the following holds 
for every $t_0 \in [0,T)$: for any  smooth strictly convex 
solution to the $\alpha$-GCF $\Sigma'_t=\p \mathcal{K}_t' $ with $\mathcal{K}_{t_0}' \subset \mathcal{K}_{t_0}$, the comparison $\mathcal{K}'_t \subset \mathcal{K}_t$ holds  for all $t\in[ t_0,T)$. Similarly, $\Sigma_t =\p\mathcal{K}_t$ is a 
{\em  viscosity supersolution}  to the $\alpha$-GCF if the following holds  for every $t_0 \in [0,T)$: for any 
smooth strictly convex  solution to the $\alpha$-GCF $\Sigma'_t=\p \mathcal{K}_t' $  with $\mathcal{K}_{t_0}\subset \mathcal{K}_{t_0}'$, the comparison  $\mathcal{K}_t\subset \mathcal{K}_t'$  holds  for all $t\in[ t_0,T)$.  
$\Sigma_t$ is a {\em viscosity solution}  if it is both a viscosity subsolution and a viscosity supersolution. 
\end{definition}

{ We state} the existence and uniqueness of a viscosity solution starting at any convex hypersurface $\Sigma_0= \p \mathcal{K}_0 \subset \R^{n+1}$,   compact or non-compact,  and asymptotic to a cylinder. {Its proof is rather a straightforward application of standard smooth approximations and the comparison principle. }

\begin{theorem} \label{thm-exist1} Let $\Sigma_0=\p \mathcal{K}_0 \subset \R^{n+1}$ be  a convex hypersurface. If $\Sigma_0$ is compact, then there is a unique viscosity solution $\Sigma_t$ to the $\alpha$-GCF running from  $\Sigma_0$ and defined over $t \in [0,T)$ for some $T <+\infty$. If $\Sigma_0$ is {non-compact} and asymptotic to a cylinder {$\partial\Omega\times \mathbb{R}$} then there is a unique viscosity solution $\Sigma_t$  to the $\alpha$-GCF running from $\Sigma_0$ defined for all $t\in [0,+\infty)$. Moreover, $\Sigma_t$ is non-compact and  asymptotic to $\Omega\times \mathbb{R}$  for all $t\in[0,\infty)$. 
\end{theorem}
\begin{proof}
Consider the first case that $\Sigma_0$ is {\em compact.} 
Choose an increasing sequence {of} convex bodies $\mathcal{K}_{i,0}$ with smooth strictly convex boundaries  $\Sigma_{i,0}$  (see Ch3.4 \cite{Schneider} for an approximation by smooth strictly convex hypersurfaces) {satisfying} $\cup_i \mathcal{K}_{i,0} = \text{int}(\mathcal{K}_0)$. Let  $\Sigma_{i,t} $ be the unique  smooth solution to the $\alpha$-GCF starting from $\Sigma_{i,0}$ (see in \cite{Ch85}). 
By the comparison principle the sequence $ \Sigma_{i,t}$ is increasing in $i$, {and hence}  the limits  $\mathcal{K}_t:= \cup_i \mathcal{K}_{i,t}$ and $\Sigma_{t}=\p\mathcal{K}_t$ exist. 
We claim that $\Sigma_t$ is a viscosity solution with initial data $\Sigma_0$. By {the} construction, $\Sigma_t$ is a viscosity supersolution.

Let us next show  that $\Sigma_t$ is a  viscosity subsolution as well. Assume  without loss of generality that $0\in \mathcal{K}_0$ {and $t_0=0$. Let $\Sigma_t'=\partial \mathcal{K}_t$ be }a smooth strictly convex $\alpha$-GCF flow with $\mathcal{K}_0' \subset \mathcal{K}_0$. 
Given {a} small $\epsilon >0$, we  consider  the rescaled solution  $(1-\epsilon) \, \Sigma'_{\tau} $,  with $\tau =(1-\e)^{ -(1+n\alpha)}\, t$ starting at  $(1-\epsilon) \, \Sigma'_0$. If $i_0$ is sufficiently  large,  $(1-\e) \mathcal{K}'_0  \subset \mathcal{K}_{i,0}$ holds if $i \geq i_0$. Thus,  the comparison principle
guarantees  $(1-\e) \mathcal{K}'_{\tau} \subset \mathcal{K}_{i,t}$.  Taking the limit  $i\to \infty$ and passing $\e\to 0$ yield the inclusion $\mathcal{K}_t' \subset \mathcal{K}_t${. This proves} that  $\Sigma_t$ is a viscosity  subsolution. We then conclude that  $\Sigma_t$ is a viscosity  solution.

For the uniqueness assertion, let us assume  that we have another viscosity solution $\Sigma''_t$ starting at $\Sigma_0$. Then, the same argument  as above, shows that each small $\e>0$, there is $i_0 \gg 1$ such that $$(1-\e) \mathcal{K}_{i,(1-\e)^{-(1+n\alpha)} t} \subset \mathcal{K}''_{t} \subset (1+\e) \mathcal{K}_{i,(1+\e)^{-(1+n\alpha)} t}, \quad  \text{ for }i\ge i_0.$$ Taking the limit   $i\to \infty$ and passing $\e \to 0$, we conclude that $\Sigma_t = \Sigma_t''$. {The finiteness} of $T$ follows by comparing the solution with a huge spherical solution containing it.  

\bigskip

Consider the next case that $\Sigma_0$ is {\em non-compact and asymptotic to}  $\Omega\times \mathbb{R}$. Choose a sequence of increasing compact sets  $\mathcal{K}_{i,0}$ with smooth strictly convex boundaries $\Sigma_{i,0}$ such that  $\cup_i \mathcal{K}_{i,0}= \mathcal{K}_0$. Let $\Sigma_{i,t}$, $t\in[0,T_i)$, be the unique smooth strictly convex solutions to the $\alpha$-GCF and define $\mathcal{K}_t := \cup_{i} \mathcal{K}_{i,t}$ and $\Sigma_{i,t}=\p\mathcal{K}_{i,t}$ as before. Note $\Sigma_{i,t}$ exists for $t\in[0,T)$, where $T\ge \lim _{i\to\infty} T_i =(0,\infty]$. By the construction, $\Sigma_t$ is already a viscosity supersolution. Let $\Sigma'_t= \p \mathcal{K}_t'$ be a smooth strictly convex $\alpha$-GCF with $\mathcal{K}'_0\subset \mathcal{K}_0$. When $\mathcal{K}'_0$ is compact, one can use the same argument as before to show $\mathcal{K}'_t\subset \mathcal{K}_t$. Let us assume $\mathcal{K}'_0$ be   non-compact. Then $\mathcal{K}'_0$ has to be asymptotic to a cylinder $\Omega'\times \mathbb{R}$ with $\Omega'\subset \Omega$.  By the same scaling and limiting argument, we may assume $\Omega'\subset\subset \Omega$. For such a $\Sigma'_0$, {\cite{CDKL} shows that there exists a unique smooth solution (thus it is equal to $\Sigma'_t$ by the uniqueness)} for all $t\in[0,\infty)$ and the solution is written on the fixed domain $\Omega'$. Moreover, the construction in \cite{CDKL} shows $\mathcal{K}'_t$ can be approximated by an increasing sequence of compact smooth strictly convex $\alpha$-GCFs $\Sigma'_{i,t}=\p \mathcal{K}'_{i,t}$. For each $\mathcal{K}'_{j}$, there is $i_j$ such that $\mathcal{K}'_{j,0} \subset \mathcal{K}_{i_j,0}$. This implies $\mathcal{K}'_{j,t} \subset \mathcal{K}_{i_j,t} \subset \mathcal{K}_t$ and proves $\mathcal{K}'_t\subset \mathcal{K}_t$. i.e. $\mathcal{K}_t$ is a viscosity solution. This also shows $T=\infty$ since we may put a non-compact rotationally symmetric strictly convex hypersurface which is asymptotic to a round cylinder in the inside of $\mathcal{K}_0$ and apply the comparison principle. Finally, the cylinder asymptotic to $\mathcal{K}_t$  does  not shrink {along the flow}  since we can insert such a barrier arbitrarily close to the boundary of $\Omega\times \mathbb{R}$ at {the} initial time $t=0$.
  \end{proof}
  
  \begin{corollary} Let $\mathcal{K}_0$ be either a bounded convex body or an unbounded convex body which is contained in a bounded cylinder.  If $\mathcal{K}_{i,0}$ is an increasing sequence of convex bodies such that ${\partial}(\cup_{i,0}\mathcal{K}_{i,0})={\partial}  \mathcal{K}_0$, then ${\partial}(\cup_{i}\mathcal{K}_{i,t})={\partial} \mathcal{K}_t$.   Here, ${\partial} \mathcal{K}_{i,t}$ and ${\partial} \mathcal{K}_t$ are the viscosity $\alpha$-{GCFs} running from $\p \mathcal{K}_{i,0}$ and $\p \mathcal{K}_{0}$, respectively.   \end{corollary}
  
In this paper, when $\Sigma_{i,t}=\p \mathcal{K}_{i,t}$ is referred, it means approximating smooth compact strictly convex solutions of $\Sigma_t$ from inside unless otherwise stated.

\section{Local speed estimate} \label{sec-harnack}
We review the following Harnack estimate which was shown by B. Chow in \cite{Ch91}. 
\begin{theorem}[B. Chow \cite{Ch91}]\label{thm-Chow} {For a} smooth compact strictly convex solution to the $\alpha$-{\em GCF} with $\alpha>0$, there holds   
\be\label{eqn-harnK}
\fr{1}{ K^\alpha}(\p_t K^\alpha - b^{ij}\nabla_iK^\alpha \nabla_j K^\alpha  )\ge -\fr{n\alpha}{1+n\alpha} \fr{1}{t}.
\ee
\end{theorem}
This has the following consequence:
\begin{prop}\label{prop-utt} Let  $x_{n+1}=u(x',t)$ be a smooth strictly convex graphical solution to the $\alpha$-{\em GCF} with $\alpha >0$ 
over {some domain $\Omega'\subset \mathbb{R}^n$} and assume it is part of a compact smooth solution or a smooth limit of such solutions.  
Then, \be\label{eqn-utt22}
 u_{tt} \ge - \fr{n\alpha }{1+n\alpha}\fr{u_t}{t}\ee and hence, for $t_2\ge t_1>0$, 
 \be\label{eqn-u123}
 u_t({\cdot,t_2})\ge \left( \fr{t_1}{t_2}\right)^{\fr{n\alpha}{1+n\alpha}}u_t(\cdot,t_1).\ee
\begin{proof}
For any $1$-form $V_i$, $K^\alpha b^{ij} (V_i+\nabla_i\log K^\alpha)(V_j +\nabla_j \log K^\alpha)\ge0$ and the Harnack imply 

\[\p_t \log K^\alpha +2K^\alpha b^{ij}V_i \na_j \log K^\alpha  + K^\alpha b^{ij} V_iV_j \ge \fr{-n\alpha}{1+n\alpha}\fr{1}{t}.\] 

In other words, for any vector field $U^i=K^\alpha b^{ij}V_j$, \[\p_t K^\alpha +2 U^i \na_i K^\alpha + h_{ij} U^iU^j \ge \fr{-n\alpha}{1+n\alpha}\fr{K^\alpha}{t}.\]

For a graphical solution of $\alpha$-GCF, $x_{n+1} = u(x',t)$, note that $\p_t u = \fr{K^\alpha}{\la -\nu,e_{n+1}\ra}$ and

 $$\p^2_{tt} u(x',t) = (\p_t + W^i \na_i) (\fr{K^\alpha}{ \la -\nu,e_{n+1} \ra}) \quad\text{with}\quad W = \fr{K^\alpha }{\la -\nu,e_{n+1}\ra}e_{n+1}^{\text{tan}}.$$

 Here $e_{n+1}^{\text{tan}} = e_{n+1} - \la e_{n+1}, \nu \ra \nu $. Using this and $\p_t \nu = \nabla K^\alpha$, we check 

\[\ba&(\p_t + W^i \na_i) (\fr{K^\alpha }{\la -\nu,e_{n+1}\ra}) 
\\&=\fr{1}{\la -\nu,e_{n+1}\ra}\left(\p_t K^\alpha +W^i \na_i K^\alpha\right)+K^\alpha (\p_t + W^i \na_i)\fr{1}{\la -\nu,e_{n+1}\ra}
\\&=\fr{1}{\la -\nu,e_{n+1}\ra}\left(\p_t K^\alpha +2W^i \na_i K^\alpha+ h_{ij}W^iW^j \right)\ge\fr{-n\alpha}{n+\alpha} \fr{K^\alpha}{t\la -\nu,e_{n+1}\ra}.
\ea\] 

\end{proof}
\end{prop}

Suppose that $\Sigma_t$ is a non-compact viscosity $\alpha$-GCF asymptotic to {$\p\Omega\times\mathbb{R}$} and let $\{x_{n+1}=u_i(x',t)\}$ be the graph representation of the   lower parts of the approximating compact smooth strictly convex solutions $\Sigma_{i,t}$.  Let us denote by  $\Omega_{i,t}$ the spatial domain of $u_i(\cdot,t)$, { namely $\Omega_{i,t}$} is   the projection of $\Sigma_{i,t}\subset\mathbb{R}^{n+1}$ to  the hyperplane $\{x_{n+1} =0\}$. 
\begin{prop}\label{prop-harnack12} 

For each $\Omega'\subset \subset \Omega$ and $t_0>0$ there is $L>0$ {with the following significance}: for all $T>t_0$ there is $i_0$ so that \be\label{eqn-utK1}  \p_t u_i (x',t) = \fr{K^\alpha}{\la -\nu, e_{n+1}\ra}  \le L\quad \text{ for }(x',t)\in\Omega'\times [t_0,T] \text{ and }i>i_0.\ee Moreover,
for each $\Omega'\subset \subset \Omega$ there are positive constants $t_0$, $\delta$, $L$ with {the following significance}: for all $T>t_0$ there is $i_0$ so that \be\label{eqn-utK2}  0<\delta\le \p_t u_i (x',t) = \fr{K^\alpha}{\la -\nu, e_{n+1}\ra}  \quad \text{ for }(x',t)\in\Omega'\times [t_0,T] \text{ and }i>i_0.\ee

\begin{proof}
Let us assume, without loss of generality, that  $\Omega$ contains the origin and that the speed of the translating soliton defined on $\Omega$, call it $u_{\Omega}$,  is $\lambda$. Fix a small $\e_0\in(0,1/6)$ so that $\Omega'\subset\subset (1+\e_0)^{-\fr{1}{n\alpha}}\Omega$. Since $\mathcal{K}_t$ is asymptotic to $\Omega \times \mathbb{R}$ for all $t\ge0$ and $\cup_i \mathcal{K}_{i,t} = \mathcal{K}_t$, given $T'$ there is  $i_0$ such that if $i>i_0$ \be \label{eq-141}(1+\e_0)^{-\fr{1}{n\alpha}}\Omega\subset \subset \Omega_{i,t} \quad   \text{ for all }t\in[0,T'].\ee $T'$ is some number which will be chosen later.  

{By rescaling the flow}, if we define  $$\hat u(x') := (1+\e_0)^{-\fr{1}{n\alpha}}u_{\Omega}((1+\e_0)^{\fr{1}{n\alpha}}x')=u_{(1+\e_0)^{-\fr{1}{n\alpha}}\Omega}(x'),$$ 
then $\hat u$ is the translating soliton on $(1+\e_0)^{-\fr{1}{n\alpha}}\Omega$ which has {the} speed $(1+\e_0)\lambda$.
Similarly, we define the translating soliton $$\bar u(x') := (1-\e_0)^{-\fr{1}{n\alpha}}u_{\Omega}((1-\e_0)^{\fr{1}{n\alpha}}x')=u_{(1-\e_0)^{-\fr{1}{n\alpha}}\Omega}(x')$$ on $(1-\e_0)^{-\fr{1}{n\alpha}}\Omega$ which has {the} speed $(1-\e_0)\lambda$.  Depending on $\Sigma_0$, we may find a large $L>0$ such that $$\bar u(x')-\fr L2\le u(x',0) \,\,\, \text{ on } \Omega
\qquad \text{ and } \qquad u(x',0) \le \hat u(x')+\fr L2 \,\,\, \text{ on } (1+\e_0)^{-\fr{1}{n\alpha}}\Omega.$$
It follows that  there is an $i_0$ such that for  $i>i_0$, then \be\label{eq-143}\bar u(x')-L\le u_i(x',0)\text{ on } \Omega_{i,0}
\quad \text{ and } \quad u_i(x',0) \le \hat u(x')+ L\text{ on } (1+\e_0)^{-\fr{1}{n\alpha}}\Omega.\ee
Furthermore, by \eqref{eq-141} and \eqref{eq-143}, one can apply the comparison principle between $x_{n+1}=u_i(x',t)$ and two barriers so that we obtain, for $i>i_0$, $$\bar u(x')-L+(1-\e_0)\lambda t < u_i(x',t) <\hat u(x')+L+(1+\e_0)\lambda t \quad \text{ on } \, \Omega'\times [0,T'].$$
In particular, we have for $t\in[0,T']$ and $i>i_0$ $$0\le f(x',t):=(\hat u(x')+L+(1+\e_0)\lambda t) - u_i(x',t) \le 2(L+\e_0\lambda t)$$ and 
 $$0\le  g(x',t):=u_i(x',t) - (\bar u(x')-L+(1-\e_0)\lambda t)\le 2(L+\e_0\lambda t).$$
 
 \smallskip 
 {We first prove the upper bound} \eqref{eqn-utK1}. Choose $T'$ in \eqref{eq-141} by $T'=2T$. 
Suppose $\p_t u_i(x_0,t_1) =C$ at some $x_0\in \Omega'$ and $t_1\in [t_0,T]$. Then by \eqref{eqn-u123}, $\p_t u_i(x_0,t)\ge C\, \eta$ 
for some $\eta=\eta(\alpha,n) \in(0,1)$ and all  $t\in[t_1,2t_1]$.  We have 
\[\ba 0\le f(x_0,2t_1) &= f(x_0,t_1) + \int_{t_1}^{2t_1} \p_t f 
\le 2(L+\e_0\lambda t_1) +[ (1+\e_0)\lambda t_1 -  C\eta t_1]   \ea\] 
and hence ${\ds \p_t u_i(x_0,t_0)=C\le \fr{(1+3\e_0)\lambda}{\eta} + \fr{2L}{\eta t_1}}$, proving that the bound from above in \eqref{eqn-utK1} holds for any $t_1 \in [t_0, T]$ and $t_0$ fixed.

{Next, we prove the lower bound} \eqref{eqn-utK2}. To this end, suppose  $\p_t u_i(x_0,t_0) =c$ at some $x_0\in \Omega'$ and $t_0>0$. Provided $T'>t_0$ and $i>i_0$, \eqref{eqn-u123} implies that  $\p_t u_i(x_0,t)\le \gamma c$ for any $t\in[t_0/2,t_0]$ and some $\gamma=\gamma(\alpha,n)>1$.  Thus,  
\[\ba 0\le g(x_0,t_0) &= g(x_0,t_0/2) + \int_{t_0/2}^{t_0} \p_t g 
\le 2\left(L+\e_0\lambda \fr{t_0}{2}\right) +\left[\gamma c \fr{t_0}{2}- \left(1-\epsilon_0\right)\lambda \fr{t_0}{2} \right] \ea\] 
{implies} that for any  $\e_0<1/6$ we have 
\[c \ge \fr{1-3e_0}{\gamma} \lambda - \fr{4L}{\gamma t_0}\ge \fr{\lambda}{2\gamma} - \fr{4L}{\gamma t_0}.\]
Hence,  ${ \p_t u_i(x_0,t_0) =c \ge \fr{\lambda}{4\gamma}}$ if ${  t_0\ge \fr{16L}{\lambda}}$.
Let us choose $t_0 :=  {16L}/{\lambda} >0$. For every $T\ge t_0$, if we choose $T'=T$   the previous yields  the lower bound \eqref{eqn-utK2} for $i>i_0$.
\end{proof}
\end{prop}

On a strictly convex smooth solution $\Sigma_t$  we may define the Gaussian curvature $K$ as a function of the  normal vector $\nu$ at a point $p$,  i.e. we define  $\bar K (\nu,t):= K(p(\nu,t),t)$ where $p=p(\nu,t)$ is {the unique} point with $\nu(p)=\nu$.  {By the evolution of $\nu$ in \eqref{eq-normal}}, $\p_t \bar K^\alpha= \p_t K^\alpha - b^{ij}\nabla_i K^\alpha \nabla_j K^\alpha$. Hence Chow's  Harnack inequality  \eqref{eqn-harnK}  
implies 

\be \label{eq-124} \p_t \bar K^\alpha \ge -\fr{n\alpha}{1+n\alpha} \fr{\bar K^\alpha}{t}\ee
which, after integrated in time $t \in [t_1,t_2]$, gives  \be\label{eq-normalharnack}\bar K^\alpha(\cdot,t_2) \ge \left( \fr{t_1}{t_2}\right)^{\fr{n\alpha}{1+n\alpha}} \bar K^\alpha(\cdot, t_1). \ee

{A similar argument of} Proposition \ref{prop-harnack12}  applied to the support function $S(\cdot,t)$  instead of the height function $u(\cdot,t)$,  was  actually used by the authors in \cite[Section 2]{CCD}. We will need this result for {the current problem as well.   Following  similar  arguments  as in Proposition  \ref{prop-harnack12} and {\cite{CCD}},  we obtain  the following:
\begin{prop}\label{prop-harnack2} Let $\Sigma_{i,t}$ be a sequence of compact smooth strictly convex solutions which approximate {the} non-compact viscosity solution $\Sigma_t$ asymptotic to the cylinder $\partial\Omega \times \mathbb{R}$ of the bounded section $\Omega$.  For any small $\mu>0$, there are positive constants $t_0$, $\delta$ depending on $\Sigma_0$,and  $\mu$ {with the following significance}: for all $T>0$ there is $i_0$  such that, for ${\Sigma _{i,t}}$ with $i>i_0$, \[ \delta \le K^\alpha (p,t)  \quad\mbox{ if }\quad  t_0\le t\le T \text{, and }  \la -\nu(p,t),e_{n+1}\ra \ge \mu.\]
For given $t_0$, there is  $M$ depending on $\Sigma_0$ and $t_0$ {with the following significance}: for all $T>0$ there is $i_0$ such that, for ${\Sigma _{i,t}}$ with $i>i_0$,\[K^\alpha(p,t) \le M \quad\text{ if }\quad t_0\le t\le T\text{, and }\la -\nu(p,t),e_{n+1}\ra \ge 0.\]
\begin{proof}
Assume $\Omega\subset\mathbb{R}^n$ contains the origin. Define the support function  ${\ds \bar S(\nu,t) = \sup_{x\in \Sigma_t}\la x,\nu \ra }$. Let  $\epsilon_0 >0$. 
As in Proposition  \ref{prop-harnack12}, {let us consider} a translator on a slightly larger domain whose translator speed is $(1-\e_0)\lambda$. Here $\lambda$ is the speed of the translator $u_{\Omega}$ on $\Omega$. We can make that this translator contains  our  initial surface $\Sigma_0$ (and hence all $\Sigma_{i,0}$) by {translating} the translator in $-e_{n+1}$ direction. If $\bar S^+(\nu,t)$ {denotes} the support function of this translator outside, {then the comparison principle between support functions \cite[Lemma 2.6]{CCD} yields} 
\[\bar S(\nu,t) \le \bar S^+(\nu,t) = C+(1-\e_0) \, \lambda t \, \la \nu ,e_{n+1}\ra\quad\text{ on }{\Sigma _{i,t}}\cap \{\la -\nu,e_{n+1}\ra\ge0\}\quad \text{by some } \,\, C(\e_0,\Sigma_0,\alpha, n)>0 .\] 
On the other hand, by inserting a translating soliton of the speed $(1+\e_0) \lambda$ inside, we know that the  point $(L+(1+\e_0)\lambda t)e_{n+1}$
(for some $L >0$)  is located inside of $\Sigma_t$. Thus,  $\la F-(L+(1+\e_0)\lambda t)e_{n+1}, \nu \ra \ge0$ and hence, in terms of approximating solutions, for each $T'>0$ there is $i_0$ with  
\[-C+(1+\e_0)\lambda t \, \la \nu,e_{n+1}\ra \le \, \bar S(\nu,t) \quad \text{ on }{\Sigma _{i,t}} \text{ if }i>i_0 \quad \text{by some }\,\, C(\e_0,\Sigma_0,\alpha, n)>0.\]
In particular, if $i>i_0$, we have  $$0\le f(\nu,t):=\bar S^+(\nu,t) - \bar S(\nu,t) \le 2\, (C-\e_0\lambda t\la \nu,e_{n+1}\ra)\text{ for }t\in[0,T']$$ and 
 $$0\le  g(\nu,t):=\bar S(\nu,t)  -C+(1+\e_0)\lambda t \la \nu,e_{n+1}\ra \le 2\, (C-\e_0\lambda t\la \nu,e_{n+1}\ra)\text{ for }t\in[0,T'].$$

\smallskip 
In the meantime, note that ${\ds \p_t \bar S(\nu,t) = \bar K^\alpha(\nu,t)}$. In the estimates below, we assume $i>i_0=i_0(T')$. Let us prove the upper bound.  Given $t_0>0$, suppose that $\bar K^\alpha(\nu_0,t_0) =a$ at some $\nu_0 \in \mathbb{S}^n_- := \mathbb{S}^n \cap \{x_{n+1}\le 0\}$ and $0<t_0<T'/2$. Then \eqref{eq-normalharnack} implies that $(\bar K^\alpha)_t(\nu_0,t)\ge \eta \, a$ for $t\in[t_0,2t_0]$ {by} some $\eta=\eta(\alpha,n) \in(0,1)$.  Therefore
\[\ba 0\le f(\nu_0,2t_0) &= f(\nu_0,t_0) + \int_{t_0}^{2t_0} \p_t f 
\le 2(C+\e_0\lambda t_0) +[ (1+e_0)\, \lambda t_0 -\eta a\, t_0]   \ea\] 
implies that  the upper bound  ${\ds \bar K^\alpha(\nu_0,t)\le \Big ( \fr{(1+3\e_0)\lambda}{\eta} + \fr{2C}{\eta t_0} \Big )=:M}$, where $M$
 depends  on  $\Sigma_0$ and $t_0$. This proves the upper bound. 

Let us prove the lower bound. Given $\mu>0$, suppose $\bar K^\alpha(\nu_0,t_0) =a$ at some $(\nu_0 ,t_0)$ with $\la -\nu_0 ,e_{n+1}\ra \ge \mu>0$ and $0<t_0<T'$. Then by \eqref{eq-normalharnack}, $\bar K^\alpha(\nu_0,t)\le \gamma \, c$ for $t\in[t_0/2,t_0]$ and  some $\gamma=\gamma(\alpha,n)>1$. Hence, for $\e_0\in(0,1)$ to be chosen later, there is $C=C(\e_0,\Sigma_0,\alpha,n)$ such that 
\[\ba 0\le g(\nu_0,t_0) &= g(\nu_0,t_0/2) + \int_{t_0/2}^{t_0} \p_t g 
\le 2\Big (C+\e_0\lambda \fr{t_0}{2}\Big ) +\Big  ( \gamma a\fr{t_0}{2}- \left(1-e_0\right)\mu\lambda \fr{t_0}{2} \Big  )\ea\] 
implying that 
\[a \ge \fr{(1-\e_0)\mu-2\e_0}{\gamma} \lambda - \fr{4C}{\gamma t_0}.\]
Now  by choosing ${\ds \e_0:=\fr{\mu}{3+\mu}}$ (hence $(1-\e_0)\mu=3\e_0$) we have  \[a\ge \fr{\e_0}{\gamma} \lambda -\fr{4C}{\gamma t_0} =\fr{\mu}{3+\mu}\fr\lambda \gamma-\fr{4C}{\gamma t_0}\quad\text{ for some }C=C(\mu,\Sigma_0,\alpha,n). \]
Therefore   \[a=\bar K^\alpha(\nu_0,t_0)\ge\fr{\mu}{3+\mu}\fr\lambda {2\gamma} \qquad \text{ if } \,\,\,  t_0\ge 8C\, \fr{3+\mu}{\mu \lambda}.\]
In summary, given $\mu\in(0,1)$, there is $t_0=  8C\, \fr{3+\mu}{\mu \lambda}$ such that {if $t_0\le t \le T'$ and $\la -\nu,e_{n+1}\ra \ge \mu$ then $K^\alpha(\nu_0,t_0)\ge  \delta$  holds on ${\Sigma _{i,t}}$ with $t>i_0$ , where $\delta >0$ is some} constant depending on $\mu$, $\Sigma_0$, $\alpha$, and $n$. 
\end{proof}
\end{prop}

\section{Local convexity estimate}\label{sec:local_convex}

{This section, we prove}  estimates which give  local bounds from below on the minimum  principal curvature 
$\lambda_{\min}$ of our solution $\Sigma_t$ in terms of upper and lower bounds of the speed $K^\alpha$. The estimates allow us to pass to {the} limit of solutions and it  is important later in the proof of {the} main theorem. We need some preliminary results and we begin with simple observations  on convex graphs. 

\begin{lemma} \label{lem-height} 
Let $x_{n+1}=u(x')$ be a $C^2$ convex graph on $ \{ |x'|\le 2r\}$ and assume there is $\delta>0$ such that ${\ds \fr{K}{\la -\nu, e_{n+1}\ra}>\delta}$,
 where ${\ds \nu = \fr{ (  Du,-1) }{\sqrt {1+|Du|^2}}}$ denotes a unit  normal vector to the graph.  Then there is ${\ds C=C(\delta\, r^{-n}, n)}$ such that \[\sup_{|x'|\le r}  u- \inf_{|x'|\le r} u \le C\, r. \]
\end{lemma} 

\begin{proof}
We may assume without loss of generality that $r=1$ and that 
$$\inf_{|x'|\le 1} u =u(x'_1)=0 \quad \mbox{and} \quad L:=\sup_{|x'|\le 1} u=\sup_{|x'|=1} u=u(x'_2) > 0$$  with $|x'_1|\le 1$ and $|x'_2|=1$. Since $u$ is convex, the  set  $A :=\{ x'  \,:\, u(x')\le L \text{ with } |x'|<2\}$ is convex, $\{|x'|\le 1\}\subset A$  and $x'_2\in \p A$. This implies that $u\ge L$ on $B:=\{x' \,:\, \la x', x'_2 \ra >1 \text{ and } |x'|<2 \}$.  Also, the convexity of $u$ implies that,  for every  $x'\in B$, 
 $$|Du(x')| \ge \fr{u(x')-u(x'_1)}{|x'-x'_1|}\ge \fr L4 .$$  It follows that the  normal vectors   ${\ds \nu = \fr{ ( Du,-1 ) }{\sqrt {1+|Du|^2}}}$,  are contained in $$C:= \Big \{ v\in \mathbb{S}^n \,:\, 0\le \la v,-e_{n+1}\ra \le \fr{1}{\sqrt{1+(L/4)^2}}  \Big \}.$$ One can roughly bound $|C|\le c_nL^{-1}$. On the other hand, note $|B|=c_n'>0$ and hence  {our assumption yields} ${\ds |\nu[B]|=\int_B \fr{K}{\la -\nu, e_{n+1} \ra } dx' \ge c_n'  \delta.}$  Since $\nu[B] \subset C$, we conclude $c_n'\delta\le c_n \, L^{-1}$ or   $L\le C_n \, \delta^{-1}$. Recalling $L: = \sup_{|x'|\le 1} u$ and
 $\inf_{|x'|\le 1} u =0$, this finishes the proof of the lemma. 
\end{proof}

\medskip 

\begin{lemma} \label{lem-grad}Let $\Sigma=\p \mathcal{K}$ be a complete $C^2$ convex hypersurface in $\mathbb{R}^{n+1}$. Suppose $0\in \Sigma$, $\la -\nu(0),e_{n+1}\ra >0$, and that, around the origin, $\Sigma$ can be represented as a  convex graph over a disk  $D_{4\rho}:=\{x'\in\mathbb{R}^n\,:\, |x'|\le 4\rho\}$, for some $\rho >0$. i.e.  there is a convex function $u: D_{4\rho}  \to \mathbb{R}$ such that $$\{(x',u(x'))\,:\, x' \in D_{4\rho} \}= \{ (x',x_{n+1})\in \Sigma \,:\, \la -\nu(x),e_{n+1}\ra  >0\text{ and } x' \in D_{4\rho}\}=: \Gamma.$$ 
If we further assume that ${\ds \fr{K}{\la -\nu,e_{n+1} \ra} \ge \delta}$ on $\Gamma$  for some $\delta>0$, then there is $C=C(\delta,\rho,n)$ such that \[\la -\nu (x), e_{n+1}\ra^{-1} \le C \qquad \text{ on } \,\,  \{ x\in \Sigma \,:\, \la x, \nu(x) \ra \le \rho \text{ and }\la -\nu,e_{n+1}\ra \ge0 \} .\]

\begin{proof} We {may assume} that  $u(0)=0$.
By Lemma \ref{lem-height}, $u(x')=u(x')-u(0)\le C'\rho$ on $\{|x'-0|\le 2\rho\} $,  for some $C'(\delta,\rho,n)$. Therefore, the ball $B_{2\rho}( (C'+2)  \rho\, e_{n+1})$ is located above to $\Sigma\cap \{ \la -\nu ,e_{n+1} \ra\ge 0\}$. Hence  around this center point $x_1:=(C'+2)\rho e_{n+1}$, we have $\la x-x_1, \nu (x) \ra \ge 2\rho$,  for all $x\in \Sigma \cap \{\la -\nu,e_{n+1}\ra \ge0\}$. 
It follows that  for all $x\in \Sigma\cap  \{\la -\nu,e_{n+1}\ra \ge0\}$ satisfying $ \la x , \nu(x) \ra \le \rho$,  we have \[\ba2\rho \le \la x-x_1, \nu (x) \ra= \la x, \nu(x)\ra -\la x_1,\nu (x)\ra \le \rho -  \rho (C'+2) \la  \nu(x),e_{n+1}\ra \ea\]  which implies the desired bound  $\fr{1}{C'+2} \le \la -\nu(x),e_{n+1}\ra .$
\end{proof}

\end{lemma}

{The following proposition is obtained by combining the results above.}

\begin{prop} \label{prop-supportfunction}
Let $\Sigma=\p \mathcal{K} \subset \mathbb{R}^{n+1}$ be a $C^2$ convex hypersurface   a part of which is a convex graph $x_{n+1}=u(x')$ on convex domain $\Omega \subset \mathbb{R}^n$.  For given $x_0 =(x'_0,u(x'_0))\in \Sigma$ with $x'_0\in \Omega$, suppose that $\text{d}(x'_0, \p\Omega):=4\e$ and  ${\ds \fr{K}{\la -\nu,e_{n+1}\ra} \ge \delta>0}$ on $\{(x',u(x'))\,:\, |x'-x'_0| \le 2\e\}$. Then $$\{ x\in \Sigma\,:\, \la x-x_0,\nu(x)\ra \le \e,\, \la -\nu(x),e_{n+1}\ra \ge0\}$$ is compact and, on this set, there is $C=C(\delta,\e,n)$ such that \bee{\la -\nu(x),e_{n+1} \ra ^{-1}} \le C \qquad\text{and}\qquad |x-x_0|\le C\,  \mathrm{ diam}(\Omega) .\eee
\begin{proof} The first gradient bound follows directly from Lemma \ref{lem-height} and \ref{lem-grad}. The second is a consequence the gradient bound. 
\end{proof}
\end{prop}

\smallskip

Next, we show our convexity estimates. The proof is independent of previous propositions, but they will {be} combined in Corollary \ref{cor-main} to give {the} regularity estimates {for} the viscosity solutions asymptotic to a cylinder.

\begin{theorem}\label{thm-interiorestimate} For $\alpha>0$, let ${\Sigma}_t=F(\cdot,t)(\Sigma ^n)$ be a complete smooth strictly convex solution to the $\alpha$-GCF.  For $F_0:=F(p_0,t_0)\in \Sigma_{t_0}$, suppose there exist constants $\e$, $\delta$, $L>0$ such that %$K\in [\delta, M]$ on $\{ S_{x_0}\le \e \}= \{ \la F-x_0, \nu\ra  \le \e \} $ for $t\in[0,t_0]$.  
\[\delta\le K^\alpha(p,t)\le L \quad \text{  and } \quad |F(p,t)-F_0|\le L  \]  on $\{(p,t)\in \Sigma^n\times [0,t_0]\,:\, \la F(p,t)-F_0 ,\nu(p,t) \ra \le \e\}$. Then there is $C=C(\e,\delta, L,\alpha ,n)$ so that 
\[\lambda_{\min}^{-1}(p_0,t_0) \le C \, \big (1 + t_0^{-1}  \big ) .\]

\end{theorem}

\begin{proof}
We may assume $F_0=F(p_0,t_0)=0$. Let  $S:= \la F,\nu \ra$ be the support function. Under the $\alpha$-GCF, by \eqref{eq-metric} and \eqref{eq-bij} we have
  
\[\ba(\p_t - K b^{rs}\na^2_{rs}) b^1_1&=-\alpha K^\alpha  b^{i1}b^{j}_1(\alpha b^{kl}b^{mn}+b^{km}b^{ln}) \na_i h_{kl}\na_j h_{mn} - \alpha K^\alpha Hb^{1}_1 +(1+n\alpha)K^\alpha -2K^\alpha \\
&\le-\alpha K^\alpha b^{i1}b^{j}_1(\alpha b^{kl}b^{mn}+b^{km}b^{ln}) \na_i h_{kl}\na_j h_{mn} \ea.  \]

Define the  cut off function \[\eta := (\e - S)_+,\]
and  compute that \[(\p_t -Kb^{ij}\nabla_{ij}) \ln\eta =\fr{(n\alpha+1)K^\alpha}{\eta} - \fr{\alpha K^\alpha HS}{\eta} + \fr{\alpha K^\alpha b^{ij}\na_i \eta \na_j \eta}{\eta^2}.\]
For some $\gamma>0$ to be chosen later, {let us consider}   the  auxiliary test function  \[w:= \eta^2 b^{1}_1 e^{\gamma |F|^2} t\] 
and apply the maximum principle to bound the maximum of $\eta^2 \lambda_{\min}^{-1} e^{\gamma |F|^2} t$.  Suppose that a positive maximum of $\eta^2\lambda_{\min}^{-1} e^{\gamma |F|^2} t$ on $\Sigma \times[0,t_0]$ is obtained at $(p',t')$. At this point, choose  { local coordinates}  such that $b^{ij}={\lambda_i^{-1}} \delta^{ij}$, $\lambda_1=\lambda_{\min}$, and $g_{ij}=\delta_{ij}$  at $(p',t')$. 
A direct calculation using   \eqref{eq-support} and \eqref{eq-distsq} shows that at the maximum  point $(p',t')$ we have 

\be\ba\label{eq-maxw'} 0& \le (\p_t -\alpha K^\alpha b^{ij} \na^2_{ij}) \ln w\\ &\le 2\bigg[\fr{(n\alpha +1)K^\alpha }{\eta} - \fr{\alpha K^\alpha HS}{\eta} + \fr{\alpha K^\alpha b^{ij}\na_i \eta \na_j \eta}{\eta^2}\bigg] 
\\&\quad -\fr{1}{b^{11}} \bigg[ \alpha K^\alpha b^{i1}b^{j}_1(\alpha b^{kl}b^{mn}+b^{km}b^{ln}) \na_i h_{kl}\na_j h_{mn}\bigg] + \fr{ \alpha K^\alpha b^{ij} \na_i b^{11} \na_j b^{11}}{(b^{11})^2}
\\&\quad +2(n\alpha -1)\gamma K^\alpha S - 2\gamma \alpha K^\alpha b^{ij}g_{ij} +\fr1{t'}.\ea \ee 

{Notice that $S\ge 0$ for $t\le t_0$ since $0\in \Sigma_{t_0}$ and also  $S\le \e$ on the support of $\eta$. Therefore we may bound three terms in the inequality above as 

\[2\frac{(n\alpha+1)K^\alpha}{\eta} - 2\frac{\alpha K^\alpha H S}{\eta }+2(n\alpha -1)\gamma K^\alpha S \le C \left( \frac{1}{\eta}+\e \gamma \right) \] for some $C=C(L,\delta,n,\alpha)$.
 
 On the other hand, at this maximum point we have  \[\na \ln w = 2\fr{\na\eta}{\eta}+ \fr{\na b^{11}}{b^{11}} + \gamma \na |F|^2 =0\] and therefore for fixed $i$ (we are not summing over $i$)
\be \label{eq-maxw}\ba 2 \fr{\alpha K^\alpha b^{ii}\na_i \eta\na_i \eta}{\eta^2} &= \fr12 \alpha K^\alpha b^{ii}\left(\fr{\na_i b^{11}}{b^{11}} + \gamma \na_i |F|^2 \right)\left(\fr{\na_i b^{11}}{b^{11}} + \gamma \na_i |F|^2 \right)
\\&\le \fr{\alpha K^\alpha b^{ii}\na_i b^{11} \na_i b^{11}}{(b^{11})^2} + \gamma^2  \alpha K^\alpha b^{ii} \na_i |F|^2 \na_i |F|^2 
\\&\le  \fr{\alpha K^\alpha b^{ii}\na_i b^{11} \na_i b^{11}}{(b^{11})^2} + 4(\sup |F|^2)\gamma^2  \alpha K^\alpha b^{ii}   .\ea\ee
We use  \eqref{eq-maxw} for all $i\neq1$ and plug them into \eqref{eq-maxw'}.  Then,  there exists   $C=C(L, \delta, \alpha ,n)>0$ such that 

\be\ba \label{eq-687} 0 &\le 2 \fr{\alpha K^\alpha b^{11}\na_1 \eta \na_1 \eta}{\eta^2}+ \bigg [-\alpha K^\alpha b^{11}(\alpha b^{kl}b^{mn}+b^{km}b^{ln}) \na_1 h_{kl}\na_1 h_{mn} \bigg ]
\\ &\quad+ {\alpha K^\alpha (b^{11})^3 \na_1  h_{11} \na_1 h_{11}}+ \sum_{i\neq 1}2 {\alpha K^\alpha b^{ii} (b^{11})^2\,\na_i h_{11} \na_i  h_{11}}
\\ &\quad+C\,  \left( \frac{1}{\eta}+\e \gamma \right) -(2\gamma-4(\sup |F|^2)\gamma^2 ) \alpha K^\alpha b^{ii}g_{ii} +\fr1{t'} . 
\ea  \ee

Here, a crucial observation is the cancellation among the third order derivatives

\be \label{eq-thirdorder} \ba &-\alpha K^\alpha b^{11} b^{km}b^{ln} \na_1 h_{kl}\na_1 h_{mn} + {\alpha K^\alpha (b^{11})^3 \na_1  h_{11} \na_1 h_{11}}+ \sum_{i\neq 1}2 {\alpha K^\alpha b^{ii} (b^{11})^2\,\na_i h_{11} \na_i  h_{11}}
 \le 0 . \ea \ee

Let us choose ${\ds \gamma= \fr{1}{4(\sup |F|^2)}}$. Plugging $(\nabla_1 \eta )^2= (h_{11})^2 \la F,\nabla _1 F\ra ^2 $ and \eqref{eq-thirdorder} into \eqref{eq-687}, we obtain 

\[\ba\gamma\alpha K^\alpha b^{ii}g_{ii} &\le C (\fr{1}{\eta} + \e \gamma)  +\fr1{t'} +2\alpha K^\alpha\,  (\sup |F|^2) \,  (b^{11} \eta^2)^{-1}  .  \ea \]}
Combining  this inequality with  the bound $K^\alpha \ge \delta$, we conclude that there is  $C=C(\e,\delta,L,\sup |F|^2,\alpha , n)$ such that 

\[\ba b^{11} \le C \, \big ( \, 1+ \fr{1}{t'}+ \fr{1}{\eta} +  (b^{11} \eta^2)^{-1}}{ \big ).\ea\]
Note that $0<t'\le t_0$, $\eta\le \e$ and $1\le e^{\gamma |F|^2}\le e^{1/4}$. Hence the last bound yields  
 \[w(p',t') = \eta^2 b^{11} e^{\gamma |x|^2} t' \le C \, \big (1+t_0+ \fr{t_0^2}{w(p',t')} \big ) \] from which we conclude the bound  \[w(p',t') \le C \, t_0
 \, \big (1+\fr{1}{t_0} \big ).\] The theorem readily follows from \[w(p_0,t_0) :=\e^2 b^{11}(p_0,t_0) \, t_0 \le w(p',t').\] 
\end{proof}

\begin{cor} \label{cor-main}  Let $\Sigma_t \subset\mathbb{R}^{n+1}$ be a non-compact viscosity $\alpha$-GCF asymptotic to {$\p\Omega\times\mathbb{R}$} and $x_{n+1}=u(x',t)$, be the graphical representation of $\Sigma_t$ on $\Omega$. Then, for any  $\Omega'\subset\subset\Omega$  there exists  $t_0>0$ and $L>0$ such that \[u_i \to u\text{ in }C^\infty( \Omega'\times[t_0,\infty))\text{,\, and }\quad\fr{1}{\la -\nu,e_{n+1}\ra},\,\lambda_{\min}^{-1}, \,\lambda_{\max} \, \le\, L \quad\text{ on } \, (x',t)\in \Omega'\times[t_0,\infty).  \]
\begin{proof} Let us denote $4\e := d (\Omega', \p\Omega)>0$.  We also fix an approximating sequence $\Sigma_{i,t}$ and denote the graph {representation} of $\Sigma_{i,t} \cap \{\la -\nu,e_{n+1}\ra \ge0\}$ by ${x_{n+1}=u_i(x',t)}$. By Proposition \ref{prop-harnack12} and Proposition \ref{prop-supportfunction}, we obtain $T_0=T_0(\Sigma_0, \Omega', \alpha,n)$ with the following: for all $T>T_0$ there is $i_0$ so that for every ${x_i=}(x_i', u_i(x_i',t_0)) \in \Sigma_{i,t_0}$ with $i>i_0$, $x_i'\in \Omega'$ and $T_0\le t_0 \le T$, there hold  \be\label{eq-111}\fr{1}{\la -\nu(x),e_{n+1} \ra } \le C \qquad\text{and}\qquad |x-{x_i}|\le C\,  \mathrm{ diam}(\Omega) \,\, \,  \text{ on }\{S_{x_i}(x) \le \e \}\ee
 for some $C=C(\Sigma_0, \e, \alpha ,n)$. Meanwhile, Proposition \ref{prop-harnack2} gives upper and lower bounds of   $K^\alpha$ on the region $\Sigma_{i,t}\cap \{ \la -\nu, e_{n+1}\ra ^{-1}\le C\}$ for $T_0\le t\le T$. i.e. we have two-sided bounds of $K^\alpha$ on $\{S_{x_i}(x) \le \e \}$ for $t\in[T_0,T]$ for large $i>i_0$. Consequently, Theorem \ref{thm-interiorestimate}  gives a bound of $\lambda_{\min}^{-1}$ at $x_i\in\Sigma_{i,t_0}$ when $x_i'\in\Omega'$ and $T_0+1/2\le t_0\le T$. The bound on $\lambda_{\max}$ follows from the bounds on $\lambda_{\min}^{-1}$ and $K^\alpha$.
 
 To summarize, for $i>i_0$, the solutions $x_{n+1}=u_i(x',t)$ on $(x',t)\in \Omega' \times [T_0+1/2,T]$ have uniform bounds on $(1+|Du_i|^2)^{1/2}$, $\lambda_1(x')$, and $\lambda_n(x')$. One can use standard regularity estimates of uniformly parabolic equations to deduce that $u_i$ converges to $u$ in $C^{\infty}$ sense on the specified domain.  
 
  \end{proof}
\end{cor}

\section{Convergence to translating soliton}

In this section we give the proof of our main convergence result Theorem \ref{thm-main}. It will be based on
the following monotonicity formula   which holds on compact solutions and  
is  shown in  Corollary \ref{cor-thmmain} in  Appendix. {Recall the definition of  $P_{ij}$ given in  \eqref{eq-def pij} and 
$P:=b^{ij}P_{ij}$.}

\begin{theorem}\label{thm-entropy}  Let $\Sigma_t$ be a smooth compact closed strictly convex solution of the $\alpha$-GCF with $\alpha>0$. Then
\be\label{eq-entropyineq}\frac{d}{dt}  \int PK^\alpha dg=\int (P_{ij}P_{kl}b^{ik}b^{jl}+ (2\alpha-1)P^2)K^\alpha dg \ge \big (n^{-1} +2\alpha-1
\big ) \fr{(\int PK^\alpha dg)^2}{\int K^\alpha dg} . \ee
In particular, when $\alpha=1$ the last term is $\fr{n+1}{n \omega_n } \left(\int_{\Sigma_t} PKdg\right)^2$ where $\omega_n = |\mathbb{S}^n|= \int Kdg$. 
\end{theorem}

\begin{proof} Shown in Corollary \ref{cor-thmmain} in  Appendix. 

\end{proof}

\begin{prop}\label{prop-conventropy1} For $\alpha\ge1$, let $\Sigma_t=\p \{x_{n+1}\ge u(x',t)\}$ be a non-compact viscosity  $\alpha$-GCF  asymptotic to $\Omega\times \mathbb{R}$ for some bounded convex domain $\Omega$.  Then for every $\tau>0$ and $U\subset \subset\Omega$,  \[ \lim_{t\to \infty} \int_t^{t+\tau}\int_{\{(x',u(x',s))\,:\, x'\in U \}} P^2K^\alpha \, dg\, ds = \lim_{t\to \infty} \int_t^{t+\tau}\int_{\Sigma_s\cap \{x'\in U\}} P^2K^\alpha \, dg \, ds =0. \]  

\begin{proof}  Let us consider an approximating sequence of smooth compact strictly convex solutions (from inside) $\Sigma_{i,t}$ with an additional assumption that $\Sigma_{i.0}$ has {the} reflection symmetry about $\{x_{n+1} = i\}$. By Corollary \ref{cor-main},  $\Sigma_{i,t}$ converges locally smoothly to $\Sigma_t$ when their lower parts are viewed as graphs.

\smallskip 
The approximation of $\Sigma_t$ by  $\Sigma_{i,t}$ shown above and the positivity of $P^2K^\alpha$, imply that it  suffices to show the following statement: {\em for given $\tau>0$ and $\e>0$, there is $t_0$ such that for each $t\ge t_0$,  we have  } 
\be\label{eqn-st1}  \limsup_{i\to \infty} \int_{t}^{t+\tau} \int_{\Sigma_{i,t}} P^2K^\alpha \, dg \, ds\le \e. 
\ee
\begin{claim} \label{claim-51}
For any fixed finite time interval $[1,T]$, there is some large $i_0$ such that $K^{\alpha-1}\le C<\infty$ on $\Sigma_{i,t}$ for $i\ge i_0$, $t\in[1,T]$. The constant $C$ only depends on $\Sigma_0$.
\begin{proof}[Proof of Claim] This is by Proposition \ref{prop-harnack2} and the symmetry of $\Sigma_{i,t}$ with respect to $\{x_{n+1}= j\}$.\end{proof} 
\end{claim}

 By shifting  $t=1$ as the initial time we may assume the claim holds from time $t=0$. {Let us continue to show  \eqref{eqn-st1}.} The Harnack inequality \eqref{eqn-harnK} and {Claim \ref{claim-51}} yield  that, for any  $T>0$, there holds
\be \label{eq-110}\mathcal{J}^{(i)}(s):= \int_{\Sigma_{i,s}} PK^\alpha dg \ge - \fr{n}{1+n\alpha} \fr{1}{s}\int_{\Sigma_{i,s}} K^{\alpha} dg \ge -\fr{n\omega_n}{1+n\alpha}\fr{C}{{s}} \ee for all $i\ge i_0=i_0(T)$ and $s\in[0,T]$.

Let us choose ${\ds t_0 :=\tfrac{n\omega_n}{1+n\alpha} \tfrac{2}{2\alpha-1} \tfrac{C}{\e}}$. We have $ \mathcal{J}^{(i)}(s) \ge -\fr{2\alpha-1}{2}\e$, for all $ t_0\le s\le T$ and  $i\ge i_0=i_0(T)$. The monotonicity formula  \eqref{eq-entropyineq} gives that ${\ds \partial_t \mathcal{J}^{(i')}(t)  \geq  (2\alpha-1)\int_{\Sigma^{(i')}_t}  P^2 K^\alpha dg
}$, for all $t >0$. 

If there are $t\ge t_0$ and $i'>i_0(T)$ ($T>t+\tau$ will be determined later) such that $\int_{t}^{t+\tau} \int_{\Sigma_{i',t}} P^2K^\alpha \, dg \, ds > \e$, then
 \be \label{eq-113}\mathcal{J}^{(i')}(t+\tau)=\mathcal{J}^{(i')}(t)+ \int_t^{t+\tau}\p_s \mathcal{J}^{(i')}(s) \, ds \ge -\fr{2\alpha-1}{2}\e +  (2\alpha-1)  \int_t^{t+\tau} \int_{\Sigma_{i',s}} P^2 K^\alpha dg \, ds \ge \fr{2\alpha-1}{2}\e.  \ee 
From \eqref{eq-entropyineq}, we have \be\label{eq-blowupode} \p_s \mathcal{J}^{(i')}(s) \ge (2\alpha-1) \fr{[\mathcal{J}^{(i')}(s)]^2}{\int_{\Sigma_{i',s}} K^\alpha dg}\ge \fr{2\alpha-1}{\omega_n} \fr{[\mathcal{J}^{(i')}(s)]^2}{\sup_{\Sigma_{i',s}} K^{\alpha-1}}.\ee  Under the assumption that $K^{\alpha-1}\le C$ and $\mathcal{J}^{(i')}(t+\tau) \ge \fr{2\alpha-1}{2} \e$,  this ODE inequality blows up before finite time $T=T(\e,\alpha,C,n, t+\tau)$. If we choose this $T$ and then the argument shows there is no such $i'>i_0=i_0(T)$.

\end{proof}
\end{prop}

When $\alpha=1$, we {do not} need Claim \ref{claim-51} and the previous proof shows the following slightly general version. {This result will be used our subsequential research \cite{CCD2}.}
\begin{prop} \label{prop-later}For any $\tau>0$ and $\e>0$, there is $T(\tau,\e,n)>0$ such that the following holds: if $x_{n+1}=u(x,t)$ on $(x,t)\in \bar U\times [-T,T]$ for some bounded $\bar U  \subset \mathbb{R}^n$ is a smooth graphical convex solution to the classical GCF ({possibly incomplete}) which  is a smooth limit of (parts  of) smooth strictly convex closed solutions, then \[\int_{-\tau}^\tau \int _{\{(x,u(x,s))\,:\, x\in U\}}P^2 K \,dg ds\le \e.\]

%
%
%
%
%  is a smooth convex (incomplete) { can you explain how this incomplete solution will look like ?} graphical solution of  the classical GCF which is in $C^\infty(\bar U\times [-T,T])$ limit of (parts  of) smooth strictly convex closed solutions, then \[\int_{-\tau}^\tau \int _{\{(x,u(x,s))\,:\, x\in U\}}P^2 K \,dg ds\le \e.\]

\end{prop}

\smallskip 

Next, {we will show that the result of  Proposition \ref{prop-conventropy1} also holds} for $\alpha \in (1/2,1)$. In this range of exponents  we need to impose additional assumptions on the initial data
$\Sigma_0$.

\begin{prop}\label{prop-conventropy2}  For $\alpha\in(1/2,1)$, suppose that $\Sigma_0$ satisfies the assumptions of Theorem \ref{thm-main2}. Then the   conclusion of Proposition \ref{prop-conventropy1} holds.
\begin{proof}
By the assumptions, we have approximating compact hypersurfaces $\Sigma_{i,0}$ such that $\mathcal{N}^{(i)}(0) \le C$ and $\mathcal{J}^{(i)}(0)\ge -C$.  Since $\big ( \mathcal{N}^{(i)}(t)  \big )^\fr{\alpha}{1-\alpha}$ is concave in time (by Corollary \ref{cor-concavity}) 
 and $\p_t \mathcal{N}^{(i)} (t) = (\alpha-1) \mathcal{J}^{(i)}(t)$ (by Lemma \ref{lem-firstderivative}),  we conclude that $$\big ( \mathcal{N}^{(i)}(t) \big ) ^{\fr\alpha{1-\alpha}} \le M+Mt$$ for some $M=M(C,\alpha)>0$. Since ${ \fr{1-\alpha}{\alpha} <1}$, it follows that  $\mathcal{N}^{(i)}(t) \le (M+ Mt)^{\fr{1-\alpha}{\alpha}}$,  that is  the function $\mathcal{N}^{(i)}(t)$ has the sublinear growth rate. 

By an argument  similar  to \eqref{eq-110}, \[\mathcal{J}^{(i)}(t) \ge  - \fr{n}{1+n\alpha} \fr{1}{t}\int_{\Sigma_{i,t}} K^{\alpha} dg \ge -\fr{n}{1+n\alpha} \fr{(M+ Mt)^{\fr{1-\alpha}{\alpha}} }{t}. \] Hence there is $t_0=t_0(n,\alpha, \e, M)$ such that  $\mathcal{J}^{(i)}(t) \ge - \tfrac {2\alpha-1}{2} \e $ for all $t\ge t_0$.  

If there exist  $t \ge t_0$ and $i  $ for which
 \be\label{eq-1102} \int_{t}^{t+\tau} \int_{\Sigma_{i,t}} P^2K^\alpha dg \, ds >\e, \ee
 
then by the argument of \eqref{eq-113} we obtain $ \mathcal{J}^{(i)}(t+\tau) \ge \tfrac{2\alpha-1}{2}\e$. 

From \eqref{eq-blowupode}, we derive the following ODE inequality \[\p_s\mathcal{J}^{(i)}(s) \ge (2\alpha-1) \fr{[\mathcal{J}^{(i)}(s)]^2}{ (M+ Ms)^{\fr{1-\alpha}{\alpha}}}.\] By the sublinear growth of the denominator, it can be checked that the ODE blows up in the finite time $T=T(\e,\alpha,M, t+\tau)$ if {\bf $\mathcal{J}^{(i)}(t+\tau) \ge\tfrac{2\alpha-1}{2} \e>0$. }  Therefore we have the opposite inequality of \eqref{eq-1102} if $i$ is sufficiently large so that the maximum existence time $T_i$ of $\Sigma_{i,t}$
satisfies $T_i \ge T$.\end{proof}
\end{prop}

Next lemma shows that an $\alpha$-GCF satisfying $P\equiv 0$ is a translating soliton as like in the {result by R. Hamilton  \cite{Ha95}
for the mean curvature flow.} 

\begin{lemma}\label{lem-p=0}{For a manifold $M^n$,} let $F:M^n\times [-\e,\e]\to  \mathbb{R}^{n+1}$ be a strictly convex smooth immersion which satisfies {the $\alpha$-GCF} and \[P=\fr{1}{ \alpha K^\alpha}(\p_t K^\alpha - b^{ij}\nabla_iK^\alpha \nabla_j K^\alpha  )\equiv 0.\]Then $F(M^n,0)$ has to be a {(possibly incomplete) }translating soliton. 
\begin{proof}
First,  observe that for such a solution  the evolution of $P$ in \eqref{eq-Pevol} implies that  $P_{ij}\equiv0$. Let us define
\begin{align*}
T:= b^{ij}\nabla_i K^\alpha \nabla_j F +K^\alpha \nu.
\end{align*} Then 
\begin{align*}
\nabla_m T= \nabla_m b^{ij}\nabla_i K^\alpha \nabla_j F+b^{ij}\nabla_{im}^2K^\alpha \nabla_j F +b^{ij}\nabla_i K^\alpha (-h_{jm}\nu)+\nabla_m K^\alpha \nu +K^\alpha h_{mj}\nabla^j F.
\end{align*}Using 
\begin{align*}
0=P_{im}=\nabla_{im}^2K^\alpha -b^{kl}\nabla_k h_{im}\nabla_l K^\alpha  +K^\alpha h_{ik}h^k_m,
\end{align*} {we obtain}
\begin{align*}
\nabla_m T= -b^{ik}b^{jl}\nabla_m h_{kl}\nabla_i K^\alpha \nabla_j F+b^{ij}\nabla_j F(b^{kl}\nabla_k h_{im}\nabla_l K ^\alpha-K^\alpha h_{ik}h^k_m) +K^\alpha h_{mj}\nabla^j F=0.
\end{align*}
Namely, $T$ is a constant vector. Note that $\la T, \nu \ra = K^\alpha$ and this shows $F(M^n,0)$ is a translating soliton with a velocity $-T$. 

\end{proof}
\end{lemma}

\begin{proof}[Proof of Theorem \ref{thm-main}] 
In view of Corollary \ref{cor-main} and the standard parabolic regularity theory, for any given $\tau_i\to\infty$, we may take a further subsequence 
(which we still denote by $\tau_i$) so that \[ u(x',t+\tau_i)-\inf_{\Omega} u(x',\tau_i) \rightarrow u_\infty(x',t) \quad \text{ in }\,\, C^{\infty}_{loc} (\Omega\times (-\infty,\infty)).\] By Proposition \ref{prop-conventropy1} and Lemma \ref{lem-p=0}, $x_{n+1}=u_\infty(x',t)$ on $\Omega\times(-\infty,\infty)$ has to be a (possibly incomplete) translating soliton. It suffices to show this is actually the unique translating soliton defined on $\Omega$. i.e. $u_\infty(x',0)\equiv u_{\Omega}(x')$. 

Let us denote $u_{\infty,0}:= u_\infty(\cdot, 0)$, and the velocity of this possibly incomplete translating soliton {by} $\lambda \, e_{n+1}$. i.e. \[K^\alpha= \lambda \la -\nu,e_{n+1}\ra \Longleftrightarrow \left[ \fr{\det D^2u_{\infty,0}}{(1+|Du_{\infty,0}|^2)^{\fr{n+2}{2}}}\right]^\alpha = \lambda(1+|Du_{\infty,0} |^2)^{-1/2}\, \text{ on }\,\Omega.\]
This implies
\be\label{eq-speedlambda}\ba \lambda^{1/\alpha} |\Omega|&= \int_{\Omega} \fr{\det D^2u_{\infty,0}}{\left(\sqrt{1+|Du_{\infty,0}|^2}\right)^{n+2-\fr{1}{\alpha}}}\\
&= \int_{Du_{\infty,0}(\Omega)} \fr{1}{(\sqrt{1+|p|^2})^{n+2-\fr{1}{\alpha}}}\\
&\le\int_{\mathbb{R}^n} \fr{1}{(\sqrt{1+|p|^2})^{n+2-\fr{1}{\alpha}}} =:\Lambda(n,\alpha)<\infty \quad \text{ provided }\alpha>\fr12.\ea\ee
Note the {equality holds}  if and only if $|\mathbb{R}^n-Du_{\infty,0}(\Omega) |=0$. i.e. when $u_{\infty,0}=u_{\Omega}$; (see the characterization of $u_\Omega$ which is discussed after Definition \ref{def-u_Omega}).  

\smallskip 
Assume {without} loss of generality that $\Omega$ contains the origin. Since we can apply the previous argument  for every subsequence of the sequence $\tau_i$, this implies \be\label{eq-limsup}\limsup_{s\to\infty} u_t(0,s)\le   \left( \fr{\Lambda(n,\alpha)}{|\Omega|} \right)^\alpha=: \lambda_\Omega.\ee
 In view of the argument in the first paragraph, we can always find a converging subsequence. Thus it suffices to show \[\liminf _{s\to\infty} u_t(0,s) \ge \lambda_{\Omega}.\]
 On the contrary, suppose there is a sequence of time $\tau_i\to\infty$ such that $u_t(0,\tau_i) \le \lambda_\Omega(1-8\delta)$ for some $\delta>0$. Due to Proposition \ref{prop-utt}, there is a small $c>0$ such that $u_t(0,s)\le \lambda_\Omega (1 -4\delta)$ on $s\in[(1-c)\tau_i, \tau_i]$.
 By \eqref{eq-limsup}, for every fixed $\e>0$, $u(0,(1-c)\tau_i)\le (1+\e)\lambda_\Omega(1-c)\tau_i +O(1)$ as $i\to\infty$. Thus \[\ba u(0,\tau_i)&\le u(0,(1-c)\tau_i)+ \lambda_\Omega(1-4\delta)\, c\tau_i \\
 &{\le (1+ (\e(1-c)-4\delta c)) \lambda_\Omega \tau_i + O(1)}.\ea\] Choosing $\e:= \fr{2c\delta}{1-c}$, we get \[u(0,\tau_i) \le (1-2\delta c){\lambda_\Omega }\tau_i + O(1)\quad \text{as }i\to\infty.\]
i.e. the average speed evaluated along the sequence $\tau_i$ is strictly less than $\lambda_\Omega$.

  On the other hand, by considering {slightly slower translating soliton defined on larger domain} $$\bar u(x') := (1-\delta c)^{-\fr{1}{n\alpha}}u_{\Omega}((1-{\delta c})^{\fr{1}{n\alpha}}x')=u_{(1-{\delta c})^{-\fr{1}{n\alpha}}\Omega}(x')$$ and use $\bar u (\cdot)- L$ for large $L$ as an initial barrier as we did {in the proof of} Proposition \ref{prop-harnack12}, we obtain \[u(0,t) \ge (1-\delta c) t - O(1)\quad \text{as }t\to\infty. \] This is a contradiction and finishes the proof.

\end{proof}

\begin{proof}[Proof of Theorem \ref{thm-main2}] In this case, we assume that \[\mathcal{N}(0):=\int_{\Sigma_0} K^\alpha dg \le C \] and also that 
\[  \mathcal{J}(0):= \int_{\Sigma_t} P_{ij} b^{ij} K^\alpha \,dg \ge -C\]
for some $C<\infty$, with $ P_{ij} := \nabla^2_{ij} K^\alpha -b^{mn}\nabla_m h_{ij} \nabla_n K^\alpha + K^\alpha h^k_ih_{kj} $.  Note that for compact solutions Lemma \ref{lem-firstderivative}
gives  $\p_t \mathcal{N}= (\alpha-1) \mathcal{J}$  and hence this assumption corresponds to giving upper bounds on $\mathcal{N}$ and its first time derivative at $t=0$.  Then the proof goes same as the proof of Theorem \ref{thm-main} { except} that we use Proposition \ref{prop-conventropy2} instead of Proposition \ref{prop-conventropy1}.
\end{proof}

\appendix 
\section{Monotonicity formula}

Let $F:M^n\times[0,T]\to \mathbb{R}^{n+1}$ be a parametrization of a smooth strictly convex closed solution $\Sigma_t$ of the $\alpha$-GCF. 
We define the entropies
 \be\label{eq-N}\ds \mathcal{N}(t):= \int_{\Sigma_t} K^\alpha \, dg\ee
and 
 \be \label{eq-J} \mathcal{J}(t):= \int_{\Sigma_t} P_{ij} b^{ij} K^\alpha \,dg\ee
 where
\be  P_{ij} := \nabla^2_{ij} K^\alpha -b^{mn}\nabla_m h_{ij} \nabla_n K^\alpha + K^\alpha h^k_ih_{kj}.\ee
Here, $dg:=\sqrt {\det g} \, dx$ is  the intrinsic volume form inherited from the imbedding $F$.
In this section we will summarize and prove certain  entropy identities  and inequalities which are used in this work.

\begin{lemma}\label{lem-firstderivative}\begin{equation}
\frac{d}{dt}\mathcal{N}(t)=(\alpha-1) \mathcal{J}(t).
\end{equation}
\end{lemma}

\begin{proof}

By equation  \eqref{eq-speed} and \eqref{eq-volume}, 
\be \label{eq-11}\ba \frac{d}{dt}(K^\alpha dg) =\frac{d}{dt}(K^\alpha) \, dg + K^\alpha \frac{d}{dt} dg&= \big ( \alpha K^\alpha b^{ij} \nabla^2_{ij}K^\alpha +\alpha HK^{2\alpha}-HK^{2\alpha} \big ) \, dg\\
&= \big (  \alpha \, b^{ij}\nabla^2_{ij}K^\alpha +(\alpha-1) HK^\alpha \big ) \, K^\alpha \, dg. \ea\ee
Hence

\[\ba  \fr{d}{dt}\mathcal{N} = \int_{\Sigma_t} \big (  \alpha \, b^{ij}\nabla^2_{ij} \, K^\alpha +(\alpha-1) HK^\alpha \big ) \, K^\alpha \, dg.\ea\]
Using  the following integration by parts 
\[\ba \int K^\alpha b^{ij} \nabla^2_{ij} K^\alpha dg & =\int  K^{\alpha-1} Kb^{ij} \nabla^2_{ij} K^\alpha \, dg \\
 & = - \int b^{ij} K\nabla_i K^{\alpha-1} \nabla_j K^\alpha dg\qquad  \big (  \text{by eq } \eqref{eq-divfree} \big )\\
 & = - \int \fr{\alpha-1}{\alpha} b^{ij} \nabla_iK^\alpha \nabla_j K^\alpha \, dg \ea, \] we conclude the desired identity 
 \[\ba  \fr{d}{dt}\mathcal{N}(t) &= (\alpha-1)\int_{\Sigma_t} \Big (  b^{ij}\nabla^2_{ij}K^\alpha-\fr{b^{ij}}{\alpha K^\alpha}\na_iK^\alpha \na_j K^\alpha+HK^\alpha \Big ) \, K^\alpha \,dg\\ 
&= (\alpha-1)\, \mathcal{J}(t).\ea  \]

\end{proof}

 \begin{theorem}\label{thm-mainappendix}  \[ \frac{d}{dt} \mathcal{J}(t) =  \int_{\Sigma}  \big ( b^{ik}b^{jl}P_{ij}P_{kl}+(2\alpha-1)(b^{ij}P_{ij})^2 \big ) \, K^\alpha dg.\]

 \begin{remark} Note that 
 \be \label{eq-P}b^{ij} P_{ij} = b^{ij} \nabla^2_{ij} K^\alpha -\fr{b^{ij}}{\alpha K^\alpha}\nabla_i K^\alpha \nabla_jK^\alpha +K^\alpha H\ee 
 which follows from \eqref{eq-nablaK}.
 
 \end{remark}
 
 \begin{proof}[Proof of Theorem] The evolution of $b^{kl}P_{kl}$,   shown in Theorem 3.2  \cite{Ch91},  is given by 
\be \label{eq-Pevol}\ba\frac{d}{dt} ( b^{kl}P_{kl})& = \alpha K^\alpha b^{ij} \nabla^2_{ij} \big ( b^{kl}P_{kl} )+ 2\alpha b^{ij} \nabla_i K^\alpha \nabla_j(b^{kl}P_{kl})+ b^{ik}b^{jl}P_{ij}P_{kl} + \alpha (b^{kl}P_{kl} \big )^2.\ea\ee
By \eqref{eq-Pevol} and \eqref{eq-11}, we obtain 
 \be \label{eq-12}\ba \frac{d}{dt} \mathcal{J}&= \int \frac{d}{dt} (b^{kl}P_{kl}) K^\alpha dg + \int  \big ( b^{kl}P_{kl} \big ) \, \big (\alpha b^{ij}\nabla^2_{ij} K^\alpha +(\alpha-1)HK^\alpha \big )\, K^\alpha \, dg \\
&= \int \big ( b^{ik}b^{jl}P_{ij}P_{kl} + \alpha (b^{kl}P_{kl})^2 \big ) \, K^\alpha \, dg + I \ea\ee 
where\be\label{eq-I} \ba I : =
&\int \Big (  \alpha K^\alpha b^{ij}\nabla^2_{ij}(b^{kl}P_{kl}) + 2\alpha b^{ij} \nabla_i K^\alpha \nabla_j (b^{kl}P_{kl}) \\
&\quad\quad\quad\quad+\alpha  (b^{ij}\nabla^2_{ij}K^\alpha)(b^{kl}P_{kl})+(\alpha-1)(b^{kl}P_{kl}) HK^\alpha \Big ) \, K^\alpha \, dg. \ea\ee

\smallskip 
To finish the proof of  the theorem it suffices to show that $I=(\alpha-1) \int (b^{kl}P_{kl})^2K^\alpha dg$. Note that for any two 
functions $F$ and $G$ we have the following integration by parts  formula:  
\be\label{eq-intbypart}\ba\int  \big ( \nabla^2_{ij}G \big ) \,   \big ( b^{ij} F K^\alpha \big ) \,  dg & = -\int \nabla_j G \,  \nabla_i\, \big  ( b^{ij} F  K^\alpha \big ) \, dg \\& = -\int \Big ( \nabla_j G \,  \nabla_i\, \big  ( ( b^{ij} K ) F  K^{\alpha -1}  \big ) \, dg\\
&= - \int b^{ij} \nabla_j G \, \nabla_i F \, K^\alpha + FK b^{ij} \, \nabla_jG \, \nabla_i K^{\alpha-1}  \, dg\\
(\text{by} \,\, \nabla_i (b^{ij} K) =0 ) \quad  &=- \int b^{ij} \nabla_j G\,  \nabla_i F \, K^\alpha \, dg -\fr{\alpha-1}{\alpha} \int F b^{ij} \, \nabla_j G \, {\nabla_i K^\alpha} \, dg .\ea\ee
Applying  formula \eqref{eq-intbypart} with $F:=\alpha K^\alpha$ and $G:= b^{kl}P_{kl}$ we obtain 
\[\int   \nabla^2_{ij} (b^{kl}P_{kl}) \, \big (  \alpha K^\alpha \, b^{ij} \big ) \, K^\alpha dg = (-2\alpha+1) \int  b^{ij} \nabla_i K^\alpha \, \nabla_j (b^{kl}P_{kl})  \, K^\alpha dg. \]
Hence,  
\[\ba \int \big ( \alpha K^\alpha b^{ij}\nabla^2_{ij}(b^{kl}P_{kl}) +& 2\alpha b^{ij} \nabla_i K^\alpha \nabla_j (b^{kl}P_{kl}) \big ) \, K^\alpha \, dg \\
&=\int  b^{ij} \nabla_i K^\alpha \nabla_j (b^{kl}P_{kl})K^\alpha dg \\ 
(\text{by eq } \eqref{eq-intbypart})\quad&= \int \Big ( - (b^{kl}P_{kl}) (b^{ij}\nabla^2_{ij}K^\alpha)K^\alpha-\fr{\alpha-1}{\alpha} (b^{kl}P_{kl})b^{ij} \nabla_i K^\alpha \nabla_j K^\alpha \Big ) \, dg.\ea\]
Plugging this into \eqref{eq-I} yields  
 \[\ba  I&=
\int \Big ( \alpha K^\alpha b^{ij}\nabla^2_{ij}(b^{kl}P_{kl}) + 2\alpha b^{ij} \nabla_i K^\alpha \nabla_j (b^{kl}P_{kl})\\
&\quad\quad\quad\quad+\alpha  (b^{ij}\nabla^2_{ij}K^\alpha)(b^{kl}P_{kl})+(\alpha-1)(b^{kl}P_{kl})HK^\alpha\Big ) \, K^\alpha \, dg \\
&=\int (\alpha-1)[b^{ij}\nabla^2_{ij} K^\alpha - \fr{1}{\alpha K^\alpha} b^{ij}\nabla_iK^\alpha \nabla_j K^\alpha +HK^\alpha] (b^{kl}P_{kl})K^\alpha \, dg \\
\big (\text{by}\, \eqref{eq-P} \big ) \quad &= (\alpha-1)\int( b^{kl}P_{kl})^2 \, K^\alpha \, dg.\ea\]
This finishes the proof of the theorem.  
 \end{proof}
 \end{theorem}

\begin{corollary}\label{cor-thmmain}
For $\alpha \geq \tfrac{n-1}{2n}$, we have 

 \[\frac{d}{dt} \int (b^{ij}P_{ij})K^\alpha dg \ge \big( \frac1n +2\alpha-1\big) \int (b^{ij}P_{ij})^2 K^\alpha dg \ge  \big( \frac{1}{n} +2\alpha-1\big)  \fr{ \big(\int (b^{ij}P_{ij})K^\alpha dg\big)^2}{\int K^{\alpha} \, dg}\ge 0.\]
 
\begin{proof} The $\alpha=1$ case is proven in Lemma 4.3 \cite{Ch91}. In the more general case, the result readily follows 
by the previous Theorem, the inequality 
\[ b^{ik}b^{jl} P_{ij}P_{kl} \ge \fr{1}{n} (b^{ij}P_{ij})^2\]
and the H\"older inequality. 
\end{proof}
\end{corollary} 

 \begin{corollary}\label{cor-concavity} For $\alpha>0$ with $\alpha\neq1$, we have 
 \[\frac{d^2}{dt^2} \, \mathcal{N}^{\fr{\alpha}{1-\alpha}}\le 0. \]
 \begin{proof}
Recall that ${\ds \tfrac{d}{dt} N =(\alpha-1) \int(b^{ij}P_{ij})K^\alpha \, dg} $,  by  Lemma \ref{lem-firstderivative}. Hence 
\[\ba \tfrac{d^2}{dt^2}\mathcal{N}^{\fr{\alpha}{1-\alpha}}& = \tfrac{d}{dt} \big ( \fr{\alpha}{1-\alpha} \mathcal{N}_t \, \mathcal{N}^{\fr{2\alpha-1}{1-\alpha}}\big ) \\
 &= \tfrac{d}{dt}\big ( -{\alpha}\mathcal{J}\mathcal{N}^{\fr{2\alpha-1}{1-\alpha}}\big ) \\
 &= -\alpha \big ( (\tfrac{d}{dt}\mathcal{J})\, \mathcal{N}^{\fr{2\alpha-1}{1-\alpha}}-(2\alpha-1)\mathcal{J}^2\mathcal{N}^\fr{3\alpha-2}{1-\alpha} \big ) \\
 &\le 0\quad \text{ by Corollary \ref{cor-thmmain}.}\ea\]
 \end{proof}
 \end{corollary}

%------------------------------------------------------------------------------%

\bigskip
\bigskip
\bigskip

\centerline{\bf Acknowledgements}

\smallskip 

\noindent B. Choi has been partially supported by POSTECH new faculty grant 4.0022745.01, POSTECH Basic Science Research Institute grant 2021R1A6A1A10042944, and NSF grant DMS-1600658.

\noindent K. Choi has been partially supported by NSF grant DMS-1811267 and KIAS Individual Grant MG078901.

\noindent P. Daskalopoulos has been partially supported by NSF grants  DMS-1600658 and DMS-1900702.

\bigskip
\bigskip

\end{document}